\documentclass[12pt,a4paper]{article}
\usepackage{latexsym,amssymb,amsfonts,amsmath,amsthm,nccmath,float,enumitem,setspace,bm,graphicx,multirow}
\usepackage[dvipsnames]{xcolor}
\usepackage[hidelinks]{hyperref}
\usepackage[margin=2cm]{geometry}

\setlist{topsep=3pt,partopsep=0pt,itemsep=1pt,parsep=0pt}

\hypersetup{
    colorlinks,
    linkcolor={red!80!black},
    citecolor={blue!80!black},
    urlcolor={blue!80!black}
}

\newtheorem{theorem}{Theorem}
\newtheorem{corollary}[theorem]{Corollary}
\newtheorem{lemma}[theorem]{Lemma}

\theoremstyle{definition}

\def \deg {{\rm deg}}

\def \leq {\leqslant}
\def \geq {\geqslant}

\def \P {\mathcal{P}}
\def \E {\mathcal{E}}
\def \A {\mathcal{A}}
\def \B {\mathcal{B}}

\def \S {\mathcal{S}}
\def \SP {{\rm SP}}
\def \MMS {{\rm MMS}}
\def \NLB {{\rm NLB}}
\def \L {{\rm LL}}

\def \mod#1{{\:({\rm mod}\ #1)}}

\renewcommand{\geq}{\geqslant}
\renewcommand{\leq}{\leqslant}

\let\oldproofname=\proofname
\renewcommand{\proofname}{\rm\bf{\oldproofname}}

\author{Yanxun Chang\thanks{Department of Mathematics, Beijing Jiaotong University, Beijing, China (yxchang@bjtu.edu.cn, jlzhou@bjtu.edu.cn)}, Charles J. Colbourn\thanks{School of Computing , Informatics, and Decision Systems Engineering, Arizona State University, Tempe AZ, USA (charles.colbourn@asu.edu)}, Adam Gowty\thanks{School of Mathematical Sciences, Monash University, Vic, Australia (adam.gowty@monash.edu, danhorsley@gmail.com)},\\ Daniel Horsley\footnotemark[3], Junling Zhou\footnotemark[1]}

\title{\bf New bounds on the maximum size of Sperner partition systems}
\date{}

\begin{document}
\setstretch{1.1}
\maketitle

\begin{abstract}
An \emph{$(n,k)$-Sperner partition system} is a collection of partitions of some $n$-set, each into $k$ nonempty classes, such that no class of any partition is a subset of a class of any other. The maximum number of partitions in an $(n,k)$-Sperner partition system is denoted $\SP(n,k)$. In this paper we introduce a new construction for Sperner partition systems and use it to asymptotically determine $\SP(n,k)$ in many cases as $\frac{n}{k}$ becomes large. We also give a slightly improved upper bound for $\SP(n,k)$ and exhibit an infinite family of parameter sets $(n,k)$ for which this bound is tight.
\end{abstract}

\section{Introduction}

A \emph{Sperner family} is a family of subsets of some ground set such that no set in the family is a subset of any other. Sperner families have been extensively studied (see \cite{And}, for example). Meagher, Moura and Stevens introduced Sperner partition systems in \cite{MeaMouSte} as a natural variant of Sperner families. An \emph{$(n,k)$-Sperner partition system} is a collection of partitions of some $n$-set, each into $k$ nonempty classes, such that no class of any partition is a subset of a class of any other. Most of the research on Sperner partition systems has focussed on investigating, for a given $n$ and $k$, the maximum number of partitions in an $(n,k)$-Sperner partition system. This quantity is denoted $\SP(n,k)$. The exact value of $\SP(n,k)$ is known in the following situations.
\begin{itemize}[itemsep=1.5mm,topsep=1.5mm]
    \item
$\SP(n,k)=1$ when $k=1$ or $k \leq n < 2k$ (for then any partition has a class of size $n$ or $1$).
    \item
$\SP(n,k)=\binom{n-1}{n/k-1}$ when $k$ divides $n$ (see \cite{MeaMouSte}).
    \item
$\SP(n,k)=\binom{n-1}{\lfloor n/2 \rfloor -1}$ when $k=2$ (using the Erd\H{o}s-Ko-Rado theorem, see \cite{LiMea}).
    \item
$\SP(n,k)=2k$ when $n=2k+1$ and $k$ is even (see \cite{LiMea}).
\end{itemize}
Note that for all unsolved cases we have $k \geq 3$, $n > 2k$ and $k$ does not divide $n$.

In the unsolved cases, bounds are known on $\SP(n,k)$. Let $n$ and $k$ be positive integers such that $n \geq k$, and let $c$ and $r$ be the unique integers such that $n=ck+r$ and $r \in \{0,\ldots,k-1\}$. In \cite{MeaMouSte}, the authors show that $\SP(n,k) \leq \MMS(n,k)$ where
\[\MMS(n,k) = \frac{\binom{n}{c}}{k-r+\frac{r(c+1)}{n-c}}.\]
Note that $0 \leq \frac{r(c+1)}{n-c} \leq 1$ because $0 \leq r \leq k-1$. Using this upper bound together with Baranyai's theorem \cite{Bar}, the authors of \cite{MeaMouSte} establish that $\SP(n,k)=\MMS(n,k)=\binom{n-1}{c-1}$ when $k$ divides $n$, as stated above. Finally, they note that $\SP(n+1,k) \geq \SP(n,k)$ because it is easy to augment an $(n,k)$-Sperner partition system to obtain an $(n+1,k)$-Sperner partition system with the same number of partitions. Thus they establish a naive lower bound $\SP(n,k) \geq \NLB(n,k)$ where
\[\NLB(n,k) = \mfrac{1}{k}\mbinom{n-r}{c}.\]
Despite its naivety, $\NLB(n,k)$ has hitherto been the best lower bound known on $\SP(n,k)$ for general $n$ and $k$. In \cite{LiMea}, Li and Meagher show that $\SP(2k+1,k) \in \{2k-1,2k\}$, $\SP(2k+2,k) \in \{2k+1,2k+2,2k+3\}$ and $\SP(3k-1,k) \geq 3k-1$. They also establish an inductive lower bound by showing that $\SP(n+k,k) \geq k\cdot\SP(n,k)$ for $n \geq k \geq 2$.

In this paper we introduce a new construction for Sperner partition systems using a result of Bryant \cite{Bry}. With this we are able to establish that the upper bound $\MMS(n,k)$ is asymptotically correct in many situations where $c$ is large.

\begin{theorem}\label{t:asymptotic}
Let $n$ and $k$ be integers with $n \rightarrow \infty$, $k=o(n)$ and $k \geq 3$, and let $c$ and $r$ be the integers such that $n=ck+r$ and $r \in \{0,\ldots,k-1\}$. Then $\SP(n,k)\sim\MMS(n,k)$ if
\begin{itemize}
    \item
$n$ is even and $r \notin \{1,k-1\}$; or
    \item
$k-r \rightarrow \infty$.
\end{itemize}
\end{theorem}

Note that the lower bound $\NLB(n,k)$ only implies the result of Theorem~\ref{t:asymptotic} when $r$ is very small compared to $k$, and the result of \cite{LiMea} that $\SP(n+k,k) \geq k \cdot \SP(n,k)$ never implies Theorem~\ref{t:asymptotic} (see Lemmas~\ref{l:NLBNotTight} and \ref{l:LiMeaNotTight}). It is also worth noting that the Sperner partition systems we construct to prove Theorem~\ref{t:asymptotic} are almost uniform (see Lemmas~\ref{l:mainConstruction} and \ref{l:altConstruction}, and note that it is easy to augment an almost uniform $(n,k)$-Sperner partition system to obtain an almost uniform $(n+1,k)$-Sperner partition system with the same number of partitions).

We also prove a result which provides an implicit upper bound on $\SP(n,k)$ for $k \geq 4$. In order to state it we require some definitions. For any nonnegative integer $i$ and real number $y \geq i$, let $\binom{y}{i}$ represent $\frac{1}{i!}y(y-1)\cdots(y-i+1)$. Define, for each integer $c \geq 2$, a function $\L_c:\{0\} \cup \mathbb{R}^{\geq 1} \rightarrow \mathbb{R}^{\geq 0}$ by $\L_c(0)=0$ and, for $x \geq 1$, $\L_c(x)=\binom{q}{c-1}$ where $q$ is the unique real number such that $q \geq c$ and $\binom{q}{c}=x$. An equivalent definition for $x \geq 1$ is $\L_c(x)=\frac{c}{q-c+1}x$ where $q$ is as before.

\begin{theorem}\label{t:upperBound}
If $n$ and $k$ are integers such that $n \geq 2k+2$ and $k \geq 4$, then
\[\big\lceil (1-\tfrac{r(c+1)}{n}) \cdot \SP(n,k) \big\rceil+\L_c\left(\big\lfloor \tfrac{r(c+1)}{n}\cdot\SP(n,k) \big\rfloor\right) \leq \mbinom{n-1}{c-1},\]
where $c$ and $r$ are the integers such that $n=ck+r$ and $r \in \{0,\ldots,k-1\}$.
\end{theorem}

For fixed $n$ and $k$, the left hand side of the inequality
\begin{equation}\label{e:upperBound}
\big\lceil (1-\tfrac{r(c+1)}{n})p \big\rceil+\L_c\left(\big\lfloor \tfrac{r(c+1)}{n}p \big\rfloor\right) \leq \mbinom{n-1}{c-1}
\end{equation}
is nondecreasing in $p$ and hence there is a unique nonnegative integer $p'$ such that \eqref{e:upperBound} holds for each $p \in \{0,\ldots,p'\}$ and fails for each integer $p > p'$. This $p'$ is an upper bound for $\SP(n,k)$. We will see in Corollary~\ref{c:MMSFormUpperBound} that $p'$ is always at most $\MMS(n,k)$. In practice $p'$ can be found via a binary search, beginning with $\NLB(n,k) \leq p' \leq \MMS(n,k)$.

It is worth mentioning the connection between Sperner partition systems and \emph{detecting arrays}, which are used in testing applications to allow the rapid identification and location of faults. We can represent an $(n,k)$-Sperner partition system with $p$ partitions as an $n \times p$ array in which the $(i,j)$ entry is $\ell$ if and only if the $i$th element of the ground set is in the $\ell$th class of the $j$th partition (according to arbitrary orderings). This array is then a \emph{$(1,\overline{1})$-detecting array} (see \cite{ColFanHor}) because it has the property that for any $j_1,j_2 \in \{1,\ldots,p\}$ and $\ell_1,\ell_2 \in \{1,\ldots,k\}$, the set of rows in which the symbol $\ell_1$ appears in column $j_1$ is not a subset of the set of rows in which the symbol $\ell_2$ appears in column $j_2$. (Intuitively this condition means that the ``signature'' of any one possible fault cannot be masked by the signature of any other.) So $\SP(n,k)$ can equivalently be interpreted as the maximum number of columns in a $(1,\overline{1})$-detecting array with $n$ rows and $k$ symbols.

This paper is organised as follows. In the next section we introduce some of the notation and results we require. In Section~\ref{s:mainConstruction} we detail the main construction we use to prove Theorem~\ref{t:asymptotic} and establish that it asymptotically matches the upper bound of $\MMS(n,k)$ when $c$ is large and $r \neq k-1$. The proof of Theorem~\ref{t:asymptotic} is completed in Section~\ref{s:asymptoticProof} using a variant of our main construction. We then move on to prove Theorem~\ref{t:upperBound} in Section~\ref{s:upperBoundProof} and to exhibit an infinite family of parameter sets for which the upper bound implied by Theorem~\ref{t:upperBound} is tight in Section~\ref{s:infFamily}. Finally, in Section~\ref{s:smallParams}, we conclude by examining the performance of our bounds for small parameter sets.

\section{Preliminaries}\label{s:prelims}

For integers $n$ and $k$ with $n \geq k \geq 1$ we define $c=c(n,k)$ and $r=r(n,k)$ as the unique integers such that $n=ck+r$ and $r \in \{0,\ldots,k-1\}$. We use these definitions of $c(n,k)$ and $r(n,k)$ throughout the paper and abbreviate to simply $c$ and $r$ where there is no danger of confusion. We also use $n=ck+r$ frequently and tacitly in our calculations.

An $(n,k)$-Sperner partition system is said to be \emph{almost uniform} if each class of each of its partitions has cardinality in $\{\lfloor \frac{n}{k} \rfloor, \lceil \frac{n}{k} \rceil \}$ and hence each partition has $k-r$ classes of cardinality $c$ and $r$ classes of cardinality $c+1$. For nonnegative integers $x$ and $i$, we denote the $i$th falling factorial $x$ by $(x)_i$. For a set $S$ and a nonnegative integer $i$, we denote the set of all $i$-subsets of $S$ by $\binom{S}{i}$.

A \emph{hypergraph} $H$ consists of a vertex set $V(H)$ together with a set $\E(H)$ of edges, each of which is a nonempty subset of $V(H)$. We do not allow loops or multiple edges. A \emph{clutter} is a hypergraph none of whose edges is a subset of another. A clutter is exactly a Sperner family, but we use the term clutter when we wish to consider the object through a hypergraph-theoretic lens. A set of edges of a hypergraph is said to be \emph{$i$-uniform} if each edge in it has cardinality $i$, and a hypergraph is said to be $i$-uniform if its entire edge set is $i$-uniform.

In this paper, an \emph{edge colouring} of a hypergraph is simply an assignment of colours to its edges with no further conditions imposed. Let $\gamma$ be an edge colouring of a hypergraph $H$ with colour set $C$. For each $c \in C$, the set $\gamma^{-1}(c)$ of edges of $H$ assigned colour $c$ is called a \emph{colour class} of $\gamma$. For each $c \in C$ and $x \in V(H)$, we denote by $\deg^{\gamma}_c(x)$ the number of edges of $H$ that are incident with the vertex $x$ and are assigned the colour $c$ by $\gamma$. Further, for a subset $Y$ of $V(H)$, we say that $\gamma$ is \emph{almost regular on $Y$} if $|\deg^{\gamma}_c(x)-\deg^{\gamma}_c(y)| \leq 1$ for all $c \in C$ and $x,y \in Y$. We will make use of a result of Bryant from \cite{Bry}.

\begin{theorem}[\cite{Bry}]\label{t:almostReg}
Let $H$ be a hypergraph, $\gamma$ be an edge colouring of $H$ with colour set $C$, and $Y$ be a subset of $V(H)$ such that any permutation of $Y$ is an automorphism of $H$. There exists a permutation $\theta$ of $\E(H)$ such that $|\theta(E)| = |E|$ and $\theta(E) \setminus Y = E \setminus Y$ for each $E  \in  \E(H)$, and such that the edge colouring $\gamma'$ of $H$ given by $\gamma'(E) = \gamma(\theta^{-1}(E))$ for each $E \in \E(H)$ is almost regular on $Y$.
\end{theorem}

In fact, we will only require the following special case of Theorem~\ref{t:almostReg}.

\begin{lemma}\label{l:colouringToSystem}
Let $n$ and $k$ be integers with $n \geq k \geq 1$, let $H$ be a clutter with $|V(H)|=n$, and let $\{X_1,\ldots,X_t\}$ be a partition of $V(H)$ such that any permutation of $X_i$ is an automorphism of $H$ for each $i \in \{1,\ldots,t\}$. Suppose there is an edge colouring $\gamma_0$ of $H$ with colour set $C \cup \{\mathrm{black}\}$ (where $C$ does not contain black) such that, for each $c \in C$, $|\gamma^{-1}_0(c)|=k$ and $\sum_{x \in X_i}\deg^{\gamma_0}_c(x)=|X_i|$ for each $i \in \{1,\ldots,t\}$. Then there is an $(n,k)$-Sperner partition system with $|C|$ partitions such that the classes of the partitions form a subset of $\E(H)$.
\end{lemma}

\begin{proof}
Let $X=V(H)$. Roughly speaking, we will perform $t$ applications of Theorem~\ref{t:almostReg}, where on the $i$th application we ``correct'' the colouring on $X_i$. Formally, we will construct a sequence of edge colourings $\gamma_0,\ldots,\gamma_t$ of $H$ with colour set $C \cup \{\hbox{black}\}$ such that, for each $s \in \{0,\ldots,t\}$ and $c \in C$, $|\gamma^{-1}_s(c)|=k$, $\deg^{\gamma_s}_c(x)=1$ for each $x \in \bigcup_{i=1}^s X_i$, and $\sum_{x \in X_i}\deg^{\gamma_s}_c(x)=|X_i|$ for each $i \in \{s+1,\ldots,t\}$. Note that $\gamma_0$ satisfies the claimed conditions. Furthermore, it suffices to find an edge colouring $\gamma_t$ satisfying the required conditions. To see this note that, for each $c \in C$, the edges assigned colour $c$ by $\gamma_t$ form a partition of $X$ into $k$ nonempty classes because the properties of $\gamma_t$ guarantee that $|\gamma^{-1}_t(c)|=k$ and $\deg^{\gamma_t}_c(x)=1$ for each $x \in X$. Thus the non-black colour classes of $\gamma_t$ will induce an $(n,k)$-Sperner partition system with the desired properties (any edges coloured black are not used as partition classes of the system).

Suppose inductively that an edge colouring $\gamma_s$ satisfying the required conditions exists for some $s \in \{0,\ldots,t-1\}$. Now apply Theorem~\ref{t:almostReg} with $Y=X_{s+1}$ to $\gamma_s$, to obtain an edge colouring $\gamma_{s+1}$ of $H$. For each $c \in C$, $|\gamma_{s+1}^{-1}(c)|=|\gamma_s^{-1}(c)|=k$ and $\deg^{\gamma_{s+1}}_c(x)=\deg^{\gamma_{s}}_c(x)$ for each $x \in X \setminus X_{s+1}$. Furthermore, $\deg^{\gamma_{s+1}}_c(x)=1$ for each $c \in C$ and $x \in X_{s+1}$, because $\sum_{x \in X_{s+1}}\deg^{\gamma_{s+1}}_c(x)=|X_{s+1}|$ and $\gamma_{s+1}$ is almost regular on $X_{s+1}$. Thus $\gamma_{s+1}$ satisfies the required conditions and the result follows.
\end{proof}

The next two lemmas show that existing results in \cite{LiMea,MeaMouSte} do not suffice to establish Theorem~\ref{t:asymptotic}. Lemma~\ref{l:NLBNotTight} shows that the lower bound of $\NLB(n,k)$ only implies the conclusion of Theorem~\ref{t:asymptotic} when $r$ is very small compared to $k$, and Lemma~\ref{l:LiMeaNotTight} shows that $\SP(n+k,k) \geq k\cdot\SP(n,k)$ never implies the conclusion of Theorem~\ref{t:asymptotic}.

\begin{lemma}\label{l:NLBNotTight}
For integers $n$ and $k$ with $n > 2k$, $k \geq 3$, and $n \rightarrow \infty$, 
we have
\[\NLB(n,k) \nsim \MMS(n,k)\]
unless $k \rightarrow \infty$ and $r = o(k)$.
\end{lemma}

\begin{proof}
Note that
\[\frac{\NLB(n,k)}{\MMS(n,k)}=\frac{k-r+\frac{r(c+1)}{n-c}}{k}\left(\mfrac{(n-r)_c}{(n)_c}\right) < \frac{k-r+\frac{r(c+1)}{n-c}}{k}.\]
If $k \rightarrow \infty$, then the result follows because $r \neq o(k)$ and $\frac{r(c+1)}{n-c} \leq 1$. If $k \not\rightarrow \infty$ and $r \geq 2$, the result follows because $\frac{r(c+1)}{n-c} \leq 1$. If $k \not\rightarrow \infty$ and $r=1$, then $\frac{r(c+1)}{n-c} \leq \frac{2}{3}$ because $k \geq 3$ and $c \geq 1$, and again the result follows.
\end{proof}

\begin{lemma}\label{l:LiMeaNotTight}
For integers $n$ and $k$ with $n \geq k$, $k \geq 3$ and $n \rightarrow \infty$, we have
\[k \cdot \MMS(n,k) \nsim \MMS(n+k,k).\]
\end{lemma}

\begin{proof}
Let $c=c(n,k)$ and $r=r(n,k)$. Note that
\[\frac{k\cdot\MMS(n,k)}{\MMS(n+k,k)}=\frac{k(c+1)(k-r+\frac{r(c+2)}{n+k-c-1})(n)_c}{\big(k-r+\frac{r(c+1)}{n-c}\big)(n+k)_{c+1}} \leq \mfrac{k(c+1)}{n+k}\left(\mfrac{(n)_c}{(n+k-1)_{c}}\right) \leq \left(1-\tfrac{k-1}{k(c+2)}\right)^c,\]
where we used the fact that $\frac{r(c+2)}{n+k-c-1} \leq \frac{r(c+1)}{n-c}$ in the first inequality and the fact that $\frac{k(c+1)}{n+k} \leq 1$ in the second. Because $\frac{k-1}{k} \geq \frac{2}{3}$, the last expression can be seen to be decreasing in $c$ for $c \geq 2$ and hence at most $\frac{25}{36}$.
\end{proof}

We conclude this section with a product construction for Sperner partition systems which generalises the inductive result of Li and Meagher mentioned in the introduction.

\begin{lemma}\label{l:productConstruction}
If $m$, $n$ and $k$ are positive integers such that $m \geq k$ and $n \geq k$, then \[\SP(m+n,k) \geq k\cdot\SP(m,k)\cdot\SP(n,k).\]
\end{lemma}

\begin{proof}
Let $X$ and $Y$ be disjoint sets with $|X|=m$ and $|Y|=n$. Let $p=\SP(m,k)$ and let $\P=\{\pi_1,\ldots,\pi_p\}$ be an $(m,k)$-Sperner partition system on $X$ with $p$ partitions, where $\pi_i=\{\pi_{i,1},\ldots,\pi_{i,k}\}$ for $i \in \{1,\ldots,p\}$. Let $q=\SP(n,k)$ and let $\mathcal{Q}=\{\rho_1,\ldots,\rho_q\}$ be an $(n,k)$-Sperner partition system on $Y$ with $q$ partitions, where $\rho_j=\{\rho_{j,1},\ldots,\rho_{j,k}\}$ for $j \in \{1,\ldots,q\}$. We claim that
\[\big\{\sigma_{i,j,y}: i\in\{1,\ldots,p\}, j\in\{1,\ldots,q\}, y \in \{1,\ldots,k\}\big\}\]
where
\[\sigma_{i,j,y}=\big\{\pi_{i,z} \cup \rho_{j,z+y}: z \in \{1,\ldots,k\}\big\}\]
(with the second component of the subscripts treated modulo $k$) is an $(m+n,k)$-Sperner partition system with $kpq$ partitions. To see that this claim is true, suppose that $\pi_{i,z} \cup \rho_{j,z+y} \subseteq \pi_{i',z'} \cup \rho_{j',z'+y'}$ for some $i,i' \in \{1,\ldots,p\}$, $j,j' \in \{1,\ldots,q\}$ and $y,z,y',z' \in \{1,\ldots,k\}$. Because $X$ and $Y$ are disjoint, $\pi_{i,z} \subseteq \pi_{i',z'}$ and $\rho_{j,z+y} \subseteq \rho_{j',z'+y'}$. So, because $\P$ and $\mathcal{Q}$ are Sperner partition systems, $i=i'$, $z=z'$, $j=j'$ and, because $z=z'$, $y=y'$. This establishes the claim and hence the theorem.
\end{proof}

\section{Main construction}\label{s:mainConstruction}

The following technical lemma will be useful in our constructions. It enables us to partition the edges of certain uniform hypergraphs into triples that are ``balanced'' in some sense.

\begin{lemma}\label{l:compTriples}
Let $t$ be a positive integer, let $H$ be a nonempty $(2t)$-uniform hypergraph with $V(H)=X$, and let $Y$ be a subset of $X$. Suppose that there are nonnegative integers $e_0,\ldots,e_t$ such that
\begin{itemize}
    \item[\textup{(i)}]
$|\{E \in \E(H):|E\cap Y|=t+i\}|=|\{E \in \E(H):|E\cap Y|=t-i\}|=e_{i}$ for each $i \in \{0,\ldots,t\}$;
    \item[\textup{(ii)}]
$e_i \geq e_{i+1}+s$ for each $i \in \{0,\ldots,s-1\}$ where $s$ is the largest element of $\{0,\ldots,t\}$ such that $e_s > 0$.
\end{itemize}
For any $p \in \{0,\ldots,\lfloor \frac{1}{3}|\E(H)| \rfloor\}$, we can partition some subset $\E^*$ of $\E(H)$ into $p$ (unordered) triples such that
\begin{itemize}
    \item
$\sum_{i=1}^3|E_i \cap Y|=3t$ for each triple $\{E_1,E_2,E_3\}$; and
    \item
$|\E^*_i|=|\E^*_{-i}|$ for each $i \in \{1,\ldots,t\}$, where $\E^*_i=\{E \in \E^*:|E\cap Y|=t+i\}$.
\end{itemize}
\end{lemma}

\begin{proof}
We prove the result by induction on $|\E(H)|$. In fact, we prove a slightly stronger result in which we do not require the full strength of (ii) when $p=1$ but only that $e_0 \geq 1$ (note $|\E(H)| \geq 3$ when $p=1$). Let $s$ be the largest element of $\{0,\ldots,t\}$ such that $e_s > 0$. Let the \emph{type} of an edge $E$ of $H$ be $|E\cap Y|-t$ and the \emph{type} of a triple be the multiset $[x_1,x_2,x_3]$ where $x_1,x_2,x_3$ are the types of the three edges in the triple. If $p=0$ the result is trivial. If $p=1$, we can take a single triple of type $[-s,0,s]$, because $|\E(H)| \geq 3$ and $e_0 > 0$. So we may assume $p \geq 2$.  In each of a number of cases below we first choose some initial triples of specified types and then add the remaining triples (if any are required) by applying our inductive hypothesis to the hypergraph $H'$ formed by the unassigned edges. The edges in the initial triples can be chosen arbitrarily subject to their specified type.

\begin{center}
\begin{tabular}{l|l}
  case & initial triples \\ \hline
  $s=0$ & $[0,0,0]$ \\
  $s=1$ & $[-1,0,1]$ \\
  $s=2$, ($e_2=1$ or $p=2$) & $[-2,1,1]$ and $[2,-1,-1]$ \\
  $s=2$, $e_2\geq2$, $p \geq 3$ & $[-2,0,2]$, $[-2,1,1]$ and $[2,-1,-1]$ \\
  $s \geq 3$ odd & $[-s,i,s-i]$ and $[s,-i,i-s]$ for $i \in \{1,\ldots,\min(e_s,\lfloor\frac{p}{2}\rfloor,\frac{s-1}{2})\}$ \\
  $s \geq 4$ even & $[-s,i,s-i]$ and $[s,-i,i-s]$ for $i \in \{1,\ldots,\min(e_s,\lfloor\frac{p}{2}\rfloor,s-1)\}$ \\
\end{tabular}
\end{center}

If $s \in \{0,1,2\}$, then using (i) and (ii) it is easy to confirm that we can choose triples of the types listed and then apply our inductive hypothesis to find the rest of the triples, so assume $s \geq 3$. For each $i \in \{-s,\ldots,s\}$, let $d_i$ be the number of edges of type $i$ that are in the initial triples. Let $b=\frac{s-1}{2}$ if $s$ is odd, let $b = s-1$ if $s$ is even, and let $b'=\min(e_s,\lfloor\frac{p}{2}\rfloor)$.
\begin{itemize}
     \item
If $b' > b$, then $d_0=0$, $d_{-s}=d_s=b$ and $d_i=\frac{2b}{s-1}$ for each $i \in \{-s+1,\ldots,s-1\} \setminus \{0\}$. Using this fact, along with (i) and (ii), it can be confirmed that we can choose triples of the types listed and then apply our inductive hypothesis to find the rest of the triples.
    \item
If $b' \leq b$, then $d_0=0$, $d_{-s}=d_s=b'$ and $d_i \in \{\lfloor \frac{2b'}{s-1} \rfloor, \lceil \frac{2b'}{s-1} \rceil\}$ for each $i \in \{-s+1,\ldots,s-1\} \setminus \{0\}$.  Using this fact, along with (i) and (ii), it can be confirmed that we can choose triples of the types listed and then apply our inductive hypothesis to find the rest of the triples. To see this, note the following.
\begin{itemize}
    \item
If $e_s \leq \lfloor\frac{p}{2}\rfloor$, then $H'$ contains no edges of type $s$ or $-s$, so the condition (ii) required to apply our inductive hypothesis is weaker. Because of this, the fact that $|d_i-d_j| \leq 1$ for $i,j \in \{0,\ldots,s-1\}$ is sufficient to establish this condition.
    \item
If $\lfloor\frac{p}{2}\rfloor < e_s$, then we only require one further triple and so the fact that $e_0 \geq 1$ suffices to establish our inductive hypothesis. \qedhere
\end{itemize}
\end{itemize}
\end{proof}

The next, very simple, lemma will be used to show that condition (ii) of Lemma~\ref{l:compTriples} holds in the situations in which it is applied.

\begin{lemma}\label{l:binomProduct}
Let $n$ and $t$ be positive integers such that $n \geq 6t-2$ is even, and let $e_i=\binom{n/2}{t-i}\binom{n/2}{t+i}$ for each $i \in \{0,\ldots,t\}$. Then $e_i > e_{i+1}+t$ for each $i \in \{0,\ldots,t-1\}$.
\end{lemma}

\begin{proof}
The result holds when $t=1$, so assume that $t \geq 2$. Let $i \in \{0,\ldots,t-1\}$. By routine calculation
\[e_i=\tfrac{(t+i+1)(n-2t+2i+2)}{(t-i)(n-2t-2i)}e_{i+1} \geq \tfrac{(t+1)(n-2t+2)}{t(n-2t)}e_{i+1}=\left(1+\tfrac{n+2}{t(n-2t)}\right)e_{i+1}.\]
Thus it suffices to show that $e_{i+1} \geq t^2$ because then $\tfrac{n+2}{t(n-2t)}e_{i+1} > t$. If $i \in \{0,\ldots,t-2\}$, then $e_{i+1} \geq t^2$ because $\binom{n/2}{t-i-1}\binom{n/2}{t+i+1} \geq \frac{n}{2}\cdot\frac{n}{2} \geq t^2$. Also, $e_t = \binom{n/2}{2t} \geq \binom{n/2}{2} \geq t^2$ because $n \geq 6t-2$.
\end{proof}

The following lemma encapsulates the main construction used in our proof of Theorem~\ref{t:asymptotic}. Recall that $c=c(n,k)$ and $r=r(n,k)$ are the integers such that $n=ck+r$ and $r \in \{0,\ldots,k-1\}$.

\begin{lemma}\label{l:mainConstruction}
Let $n$ and $k$ be integers such that $n \geq 2k$, $k \geq 3$, $r \neq 0$, and $n$ and $ck$ are both even. Let $u \in \{1,\ldots,\lfloor\frac{c}{2}\rfloor\}$ such that $u=\frac{c}{2}$ if $r=k-1$. There exists an almost uniform $(n,k)$-Sperner partition system with $p$ partitions where
\[p=\min\left(\left\lfloor\mfrac{a(u)}{k-r}\right\rfloor,\left\lfloor\mfrac{b(u)}{r}\right\rfloor\right),\quad a(u) = \sum_{i = u}^{c-u}\mbinom{n/2}{i}\mbinom{n/2}{c-i}, \quad
b(u) = 2\sum_{i = 0}^{u-1}\mbinom{n/2}{i}\mbinom{n/2}{c+1-i}.\]
\end{lemma}

\begin{proof}
Note that $r$ is even because $n$ and $ck$ are both even. Fix $u \in \{1,\ldots,\lfloor\frac{c}{2}\rfloor\}$ and let $a=a(u)$ and $b=b(u)$. Let $X_1$ and $X_2$ be disjoint sets such that $|X_1|=|X_2|=\frac{n}{2}$, and let $X=X_1 \cup X_2$. For each $(i,j)\in \mathbb{N} \times \mathbb{N}$, let
\[\E_{(i,j)}=\{E \subseteq X: |E \cap X_1|=i, |E \cap X_2|=j\}\]
and note $|\E_{(i,j)}|=\binom{n/2}{i}\binom{n/2}{j}$. Let
\begin{align*}
\mathcal{A}&=\medop\bigcup_{(i,j) \in I'}\E_{(i,j)}, & \mbox{where } I'&=\{(i,j) \in \mathbb{N} \times \mathbb{N}: i+j=c, \min(i,j) \geq u\} \\
\mathcal{B}&=\medop\bigcup_{(i,j) \in I''}\E_{(i,j)}, &\mbox{where } I''&=\{(i,j) \in \mathbb{N} \times \mathbb{N}: i+j=c+1, \min(i,j) \leq u-1\}.
\end{align*}
Note that $|\mathcal{A}|=a$ and $|\mathcal{B}|=b$. Furthermore, no set in $\mathcal{A}$ is a subset of a set in $\mathcal{B}$ because, for any $A \in \mathcal{A}$ and $B \in \mathcal{B}$, $|A \cap X_i| \geq u > |B \cap X_i|$ for some $i \in \{1,2\}$. So the hypergraph $H$ with vertex set $X$ and edge set $\mathcal{A} \cup \mathcal{B}$ is a clutter. Let $C$ be a set of $p$ colours other than black. Observe that, for each $i \in \{1,2\}$, any permutation of $X_i$ is an automorphism of $H$. Thus, by Lemma~\ref{l:colouringToSystem}, it suffices to find an edge colouring $\gamma$ of $H$ with colour set $C \cup \{\hbox{black}\}$ such that, for each $c \in C$, $|\gamma^{-1}(c)|=k$ and $\sum_{x \in X_i}\deg^\gamma_c(x)=\frac{n}{2}$ for each $i \in \{1,2\}$. Note that the resulting Sperner partition system will be almost uniform because each edge in $H$ has size $c$ or $c+1$. Call a set of edges $\E' \subseteq \E(H)$ \emph{compatible} if $\sum_{E \in \E'}|E \cap X_1|=\sum_{E \in \E'}|E \cap X_2|$.

\smallskip\noindent\textbf{Case 1.} Suppose that $k$ is even. Then each partition in an almost uniform $(n,k)$-Sperner partition system contains an even number, $k-r$, of classes of cardinality $c$ and an even number, $r$, of classes of cardinality $c+1$.

Because $|\E_{(i,j)}|=|\E_{(j,i)}|$ for each $(i,j) \in I''$, we can partition $\mathcal{B}$ into $\frac{b}{2}$ compatible pairs. Also, $|\E_{(i,j)}|=|\E_{(j,i)}|$ for each $(i,j) \in I'$ and, if $c$ is even, a pair of edges from $\E_{c/2,c/2}$ is compatible. Thus, we can find $\lfloor \frac{a}{2} \rfloor$ disjoint compatible pairs in $\mathcal{A}$ (one edge in $\E_{c/2,c/2}$ will be unpaired in the case where $c$ is even and $|\E_{c/2,c/2}|$ is odd, and all edges will be paired otherwise).

Take an edge colouring $\gamma$ of $H$ with colour set $C \cup \{\hbox{black}\}$ such that each non-black colour class contains $r$ edges in $\mathcal{B}$ that form $\frac{r}{2}$ compatible pairs and $k-r$ edges in $\mathcal{A}$ that form $\frac{k-r}{2}$ compatible pairs, and all remaining edges are coloured black. This can be accomplished because $\frac{r}{2}p \leq \frac{r}{2}\lfloor\frac{b}{r}\rfloor \leq \frac{b}{2}$ and $\frac{k-r}{2}p \leq \frac{k-r}{2}\lfloor\frac{a}{k-r}\rfloor \leq \lfloor \frac{a}{2} \rfloor$. Observe that for each $c \in C$ we have that $\sum_{x \in X}\deg^\gamma_c(x)=r(c+1)+(k-r)c=n$ and, because the colour class can be partitioned into compatible pairs, $\sum_{x \in X_1}\deg^\gamma_c(x)=\sum_{x \in X_2}\deg^\gamma_c(x)$. Thus, as desired, we have that $\sum_{x \in X_i}\deg^\gamma_c(x)=\frac{n}{2}$ for each $c \in C$ and $i \in \{1,2\}$.

\smallskip\noindent\textbf{Case 2.} Suppose that $k$ is odd, $c$ is even, and $r \neq k-1$. Then each partition in an almost uniform $(n,k)$-Sperner partition system contains an odd number, $k-r$, of classes of cardinality $c$ and an even number, $r$, of classes of cardinality $c+1$. Apply Lemma~\ref{l:compTriples} with $Y=X_1$, $t=\frac{c}{2}$, and $e_i=|\E_{(t-i,t+i)}|=|\E_{(t+i,t-i)}|$ for each $i \in \{0,\ldots,t\}$ to find $p$ disjoint triples of edges in $\A$. The hypotheses of Lemma~\ref{l:compTriples} can be seen to be satisfied using Lemma~\ref{l:binomProduct} and because $p \leq \lfloor\frac{a}{k-r}\rfloor \leq  \lfloor\frac{a}{3}\rfloor$ since $k-r \geq 3$. Note that each triple given by Lemma~\ref{l:compTriples} is compatible, and that the number of edges in $\E_{(i,j)}$ assigned to triples is equal to the number of edges in $\E_{(j,i)}$ assigned to triples for each $(i,j) \in I'$. Thus we can partition all, or all but one, of the unassigned edges in $\A$ into $\lfloor \frac{a-3p}{2} \rfloor$ compatible pairs. Take an edge colouring $\gamma$ of $H$ with colour set $C \cup \{\hbox{black}\}$ such that each non-black colour class contains $r$ edges in $\mathcal{B}$ that form $\frac{r}{2}$ compatible pairs and $k-r$ edges in $\mathcal{A}$ that form one compatible triple and $\frac{k-r-3}{2}$ compatible pairs, and all remaining edges are coloured black. This can be accomplished because $\frac{r}{2}p \leq \frac{r}{2}\lfloor\frac{b}{r}\rfloor \leq \frac{b}{2}$ and $\frac{k-r-3}{2}p \leq \lfloor \frac{a-3p}{2} \rfloor$ (since $\frac{k-r-3}{2}p \leq \frac{k-r}{2}\lfloor\frac{a}{k-r}\rfloor-\frac{3p}{2} \leq \frac{a-3p}{2}$ and $\frac{k-r-3}{2}p$ is an integer). Then $\gamma$ has the properties we desire.

\smallskip\noindent\textbf{Case 3.} Suppose that $k$ is odd, $c$ is even, and $r = k-1$. Then $u=\frac{c}{2}$ by our hypotheses and $\A=\E_{c/2,c/2}$. Let $\gamma$ be an edge colouring of $H$ with colour set $C \cup \{\hbox{black}\}$ such that each non-black colour class contains $k-1$ edges in $\mathcal{B}$ that form $\frac{k-1}{2}$ compatible pairs and one edge in $\mathcal{A}$. Again $\gamma$ has the properties we desire.
\end{proof}

To extend the approach of Lemma~\ref{l:mainConstruction} to cases where $n$ is even and $ck$ is odd would involve finding complementary triples of edges in $\B$. This can be difficult because the edges in $\B$ are ``unbalanced'' in terms of the sizes of their intersections with $X_1$ and $X_2$. To circumvent this problem we will introduce, in Section~\ref{s:asymptoticProof}, a variation on our construction in which the edges in $\B$ are ``balanced''. First, however, we show that, when $c$ is large and $r \neq k-1$, the lower bound implied by Lemma~\ref{l:mainConstruction} asymptotically matches the $\MMS(n,k)$ upper bound, recalling that
\[\MMS(n,k) = \frac{\binom{n}{c}}{k-r+\frac{r(c+1)}{n-c}}.\]

\begin{proof}[\textbf{\textup{Proof of Theorem~\ref{t:asymptotic} when $\bm{n}$ and $\bm{ck}$ are even.}}]
By our hypotheses, $r \neq k-1$. Furthermore, $\SP(n,k)=\MMS(n,k)$ when $r=0$, so we may assume $2 \leq r < k-1$. Let $a(j)$ and $b(j)$ be as defined in Lemma~\ref{l:mainConstruction} for each $j \in \{1,\ldots,\lfloor\frac{c}{2}\rfloor\}$, and additionally define $a(0)=\binom{n}{c}$, $b(0)=0$, $a(\lfloor\frac{c}{2}\rfloor+1)=0$, and $b(\lfloor\frac{c}{2}\rfloor+1)=\binom{n}{c+1}$.  For each $j \in \{0,\ldots,\lfloor\frac{c}{2}\rfloor+1\}$, let $a_j=\lfloor\frac{a(j)}{k-r}\rfloor$ and $b_j=\lfloor\frac{b(j)}{r}\rfloor$. Note that $a_0 \geq \cdots \geq a_{\lfloor c/2\rfloor+1}=0$, $0=b_0 \leq \cdots \leq b_{\lfloor c/2\rfloor+1}$, $a_0 > b_0$ and $a_{\lfloor c/2\rfloor+1} < b_{\lfloor c/2\rfloor+1}$. Let $w$ be the unique integer in $\{0,\ldots,\lfloor\frac{c}{2}\rfloor\}$ such that $a_{w+1} \leq b_{w+1}$ and $a_{w} > b_{w}$. By applying Lemma~\ref{l:mainConstruction} with $u=w+1$ (or trivially if $w=\lfloor\frac{c}{2}\rfloor$) we have $\SP(n,k) \geq a_{w+1}$, and by applying Lemma~\ref{l:mainConstruction} with $u=w$ (or trivially if $w=0$) we have $\SP(n,k) \geq b_{w}$. Furthermore, one of these bounds is the best bound achievable via Lemma~\ref{l:mainConstruction} because $a_{w+1} \geq \cdots \geq a_{\lfloor c/2\rfloor+1}$ and $b_0 \leq \cdots \leq b_{w}$. By definition of the function $a$, we have $a(w+1)=a(w)-\delta\binom{n/2}{w}\binom{n/2}{c-w}$, where $\delta=2$ if $w<\frac{c}{2}$ and $\delta=1$ if $w=\frac{c}{2}$.
Hence
\begin{equation}\label{e:jumpSize}
\SP(n,k) \geq a_{w+1} = \left\lfloor\frac{a(w)-\delta\binom{n/2}{w}\binom{n/2}{c-w}}{k-r}\right\rfloor \geq a_{w}-\frac{\delta\binom{n/2}{w}\binom{n/2}{c-w}}{k-r}-1.
\end{equation}

We will bound $a_{w}$ and then apply \eqref{e:jumpSize}. We now show that
\begin{equation}\label{e:coverage}
(c+1)b(w)=(n-c)\left(\mbinom{n}{c}-a(w)\right)-\delta'(n-2w+2)\mbinom{n/2}{w-1}\mbinom{n/2}{c-w+1},
\end{equation}
where $\delta'=1$ if $w \geq 1$ and $\delta'=0$ if $w=0$. We may assume $w \geq 1$, for otherwise $w=0$, $b(w)=0$, $a(w)=\binom{n}{c}$ and \eqref{e:coverage} holds. Now apply
Lemma~\ref{l:mainConstruction} with $u=w$, let $\A$ and $\B$ be as defined in its proof, and let $\A^c=\binom{X}{c} \setminus \A$. Note that $|\A|=a(w)$, $|\B|=b(w)$ and $|\A^c|=\binom{n}{c}-a(w)$. We now count, in two ways, the number of pairs $(S,B)$ such that $S \in \A^c$, $B \in \B$ and $S \subseteq B$.
\begin{itemize}
    \item
Each of the $b(w)$ sets in $\B$ has exactly $c+1$ subsets in $\binom{X}{c}$ and each of these is in $\A^c$, because no set in $\A$ is a subset of a set in $\B$.
    \item
By the definition of $\mathcal{A}$, $\min(|S \cap X_1|,|S \cap X_2|) \leq w-1$ for each $S \in \A^c$. Each of the $\binom{n}{c}-a(w)$ sets in $\A^c$ has $n-c$ supersets in $\binom{X}{c+1}$.  For each $S \in \A^c$ such that $\min(|S \cap X_1|,|S \cap X_2|) \leq w-2$, all of these supersets of $S$ are in $\B$. For each of the $2\binom{n/2}{w-1}\binom{n/2}{c-w+1}$ sets $S \in \A^c$ such that $\min(|S \cap X_1|,|S \cap X_2|) = w-1$, exactly $\frac{n}{2}-w+1$ of these supersets of $S$ are not in $\B$.
\end{itemize}
Equating our two counts, we see that \eqref{e:coverage} does indeed hold.

Because $a_w > b_w$, we have $\lfloor\frac{a(w)}{k-r}\rfloor > \lfloor\frac{b(w)}{r}\rfloor$ which implies $\frac{a(w)}{k-r} > \frac{b(w)}{r}$ or equivalently $b(w) < \frac{r}{k-r}a(w)$. Substituting this into \eqref{e:coverage} and solving for $a(w)$ we see
\begin{equation}\label{e:coverage2}
a(w)>\frac{(k-r)\big(\binom{n}{c}-\frac{\delta'(n-2w+2)}{n-c}\binom{n/2}{w-1}\binom{n/2}{c-w+1}\big)}{k-r+\frac{r(c+1)}{n-c}}.
\end{equation}
Using $a_w > \frac{a(w)}{k-r}-1$ and \eqref{e:coverage2}  in \eqref{e:jumpSize} we obtain
\[
\SP(n,k)>\frac{\binom{n}{c} -\frac{\delta'(n-2w+2)}{n-c}\binom{n/2}{w-1}\binom{n/2}{c-w+1}}{k-r+\frac{r(c+1)}{n-c}} -\frac{\delta\binom{n/2}{w}\binom{n/2}{c-w}}{k-r}-2,
\]
or, equivalently,
\begin{equation}\label{e:boundDiff2}
\SP(n,k)>\frac{\binom{n}{c}- \tfrac{\delta'(n-2w+2)}{n-c}\binom{n/2}{w-1}\binom{n/2}{c-w+1} -\delta\big(1+\frac{r(c+1)}{(n-c)(k-r)}\big)\binom{n/2}{w}\binom{n/2}{c-w}}{k-r+\frac{r(c+1)}{n-c}} -2.
\end{equation}
In the above, note that $\delta' \leq 1$, $\delta  \leq 2$, $\tfrac{n-2w+2}{n-c} \leq \frac{3}{2}$ when $w \geq 1$ because $k \geq 3$, and $\tfrac{r(c+1)}{(n-c)(k-r)} \leq 1$ because $r \leq k-1$.

Note that $\binom{n/2}{x}\binom{n/2}{c-x} \leq \binom{n/2}{\lfloor c/2 \rfloor}\binom{n/2}{\lceil c/2 \rceil}$ for any $x \in \{0,\ldots,c\}$. By using this fact and then applying Stirling's approximation we have, for $n \rightarrow \infty$ with $k=o(n)$ and any $x \in \{0,\ldots,c\}$,
\[\mbinom{n/2}{x}\mbinom{n/2}{c-x}\Big/\mbinom{n}{c} \leq \sqrt{\tfrac{2n}{\pi c(n-c)}}(1+o(1)) \leq \sqrt{\tfrac{2k}{\pi c(k-1)}}(1+o(1)) = o(1)\]
(note that $n \rightarrow \infty$ with $k=o(n)$ implies $c \rightarrow \infty$). Applying this fact twice in \eqref{e:boundDiff2} yields $\SP(n,k)>\MMS(n,k)(1-o(1))$.  Combined with the fact that $\SP(n,k) \leq \MMS(n,k)$, this establishes the result.
\end{proof}

\section{Proof of Theorem~\ref{t:asymptotic}}\label{s:asymptoticProof}

As discussed after Lemma~\ref{l:mainConstruction}, we require a variation on our main construction in order to complete the proof of Theorem~\ref{t:asymptotic}.

\begin{lemma}\label{l:altConstruction}
Let $n$ and $k$ be integers such that $n \geq 2k$, $k \geq 3$, $n$ is even and $ck$ is odd. Let $u \in \{\frac{c+1}{2}, \ldots, c - 1\}$ be such that $u=\frac{c+1}{2}$ if $r=1$. There exists an almost-uniform $(n,k)$-Sperner partition system with $p$ partitions where
\[p=\min\left(\left\lfloor\mfrac{a(u)}{k-r}\right\rfloor,\left\lfloor\mfrac{b(u)}{r}\right\rfloor\right),\quad a(u) = 2\sum_{i=u+1}^c\mbinom{n/2}{i}\mbinom{n/2}{c-i}, \quad
b(u) = \sum_{i=c+1-u}^u\mbinom{n/2}{i}\mbinom{n/2}{c+1-i}.\]
\end{lemma}

\begin{proof}
Note that each partition in an almost uniform $(n,k)$-Sperner partition system contains an even number, $k-r$, of classes of cardinality $c$ and an odd number, $r$, of classes of cardinality $c+1$.

Fix $u \in \{\frac{c+1}{2}, \ldots, c - 1\}$ and let $a=a(u)$ and $b=b(u)$. Let $X_1$ and $X_2$ be disjoint sets such that $|X_1|=|X_2|=\frac{n}{2}$, and let $X=X_1 \cup X_2$. As in the proof of Lemma~\ref{l:mainConstruction}, for each $(i,j)\in \mathbb{N} \times \mathbb{N}$, let
\[\E_{(i,j)}=\{E \subseteq X: |E \cap X_1|=i, |E \cap X_2|=j\}.\]
Unlike the proof of Lemma~\ref{l:mainConstruction}, let
\begin{align*}
\mathcal{A}&=\medop\bigcup_{(i,j) \in I'}\E_{(i,j)}, & \mbox{where } I'&=\{(i,j) \in \mathbb{N} \times \mathbb{N}: i+j=c, \max(i,j) \geq u+1\} \\
\mathcal{B}&=\medop\bigcup_{(i,j) \in I''}\E_{(i,j)}, &\mbox{where } I''&=\{(i,j) \in \mathbb{N} \times \mathbb{N}: i+j=c+1, \max(i,j) \leq u\}.
\end{align*}
Note that $|\mathcal{A}|=a$ and $|\mathcal{B}|=b$. Furthermore, no set in $\mathcal{A}$ is a subset of a set in $\mathcal{B}$ because, for any $A \in \mathcal{A}$ and $B \in \mathcal{B}$, $|A \cap X_i|>u\geq|B \cap X_i|$ for some $i \in \{1,2\}$. Thus the hypergraph $H$ with vertex set $X$ and edge set $\mathcal{A} \cup \mathcal{B}$ is a clutter. Observe that, for each $i \in \{1,2\}$, any permutation of $X_i$ is an automorphism of $H$. Let $C$ be a set of $p$ colours other than black. By Lemma~\ref{l:colouringToSystem}, it suffices to find an edge colouring $\gamma$ of $H$ with colour set $C \cup \{\hbox{black}\}$ such that, for each $c \in C$, $|\gamma^{-1}(c)|=k$ and $\sum_{x \in X_i}\deg^\gamma_c(x)=\frac{n}{2}$ for each $i \in \{1,2\}$. Note that the resulting Sperner partition system will be almost uniform because each edge in $H$ has size $c$ or $c+1$. Again, call a set of edges $\E' \subseteq \E(H)$ \emph{compatible} if $\sum_{E \in \E'}|E \cap X_1|=\sum_{E \in \E'}|E \cap X_2|$.

\smallskip\noindent\textbf{Case 1.} Suppose that $r \neq 1$. Apply Lemma~\ref{l:compTriples} with $Y=X_1$, $t=\frac{c+1}{2}$, and $e_i=|\E_{(t-i,t+i)}|=|\E_{(t+i,t-i)}|$ for each $i \in \{0,\ldots,t\}$ to find $p$ disjoint triples of edges in $\B$. The hypotheses of Lemma~\ref{l:compTriples} can be seen to be satisfied using Lemma~\ref{l:binomProduct} and because $p \leq \lfloor\frac{b}{r}\rfloor \leq  \lfloor\frac{b}{3}\rfloor$ since $r \geq 3$. Note that each triple given by Lemma~\ref{l:compTriples} is compatible, and that the number of edges in $\E_{(i,j)}$ assigned to triples is equal to the number of edges in $\E_{(j,i)}$ assigned to triples for each $(i,j) \in I''$. Thus we can partition all, or all but one, of the unassigned edges in $\B$ into $\lfloor \frac{b-3p}{2} \rfloor$ compatible pairs. Take an edge colouring $\gamma$ of $H$ with colour set $C \cup \{\hbox{black}\}$ such that each non-black colour class contains $r$ edges in $\mathcal{B}$ that form one compatible triple and $\frac{r-3}{2}$ compatible pairs and $k-r$ edges in $\mathcal{A}$ that form $\frac{k-r}{2}$ compatible pairs, and all remaining edges are coloured black. This can be accomplished because $\frac{k-r}{2}p \leq \frac{k-r}{2}\lfloor\frac{a}{k-r}\rfloor \leq \frac{a}{2}$ and $\frac{r-3}{2}p \leq \lfloor\frac{b-3p}{2}\rfloor$ (note that $\frac{r-3}{2}p$ is an integer less than or equal to $\frac{r}{2}\lfloor\frac{b}{r}\rfloor-\frac{3p}{2}$). Then $\gamma$ has the properties we desire.

\smallskip\noindent\textbf{Case 2.} Suppose that $r = 1$. Then $u=\frac{c+1}{2}$ by our hypotheses and $\B=\E_{(c+1)/2,(c+1)/2}$. Take an edge colouring $\gamma$ of $H$ with colour set $C \cup \{\hbox{black}\}$ such that each non-black colour class contains one edge in $\mathcal{B}$ and $k-1$ edges in $\A$ that form $\frac{k-1}{2}$ compatible pairs. Again $\gamma$ has the properties we desire.
\end{proof}

The approach of Lemma~\ref{l:altConstruction} can also be applied when $n$ and $k$ are both even. However, computational evidence indicates that this approach almost always underperforms Lemma~\ref{l:mainConstruction}. We can now prove the remainder of Theorem~\ref{t:asymptotic}.

\begin{proof}[\textbf{\textup{Proof of Theorem~\ref{t:asymptotic}.}}]
We saw in Section~\ref{s:mainConstruction} that Theorem~\ref{t:asymptotic} holds when $n$ and $ck$ are both even. Here, we first use Lemma~\ref{l:altConstruction} to deal with almost all of the remaining cases where $n$ is even, and then use the monotonicity of $\SP(n,k)$ in $n$ to complete the rest of the proof.

\smallskip\noindent\textbf{Case 1.} Suppose that $n$ is even, $ck$ is odd, and $r \neq 1$. The proof is very similar to the proof in the case where $n$ and $ck$ are even, but we highlight the differences.

Let $a(j)$ and $b(j)$ be as defined in Lemma~\ref{l:altConstruction} for each $j \in \{\frac{c+1}{2},\ldots,c-1\}$, and additionally define $a(\frac{c-1}{2})=\binom{n}{c}$, $b(\frac{c-1}{2})=0$, $a(c)=0$, and $b(c)=\binom{n}{c+1}-2\binom{n/2}{c+1}$. For each $j \in \{\frac{c-1}{2}, \ldots, c\}$, let $a_j=\lfloor\frac{a(j)}{k-r}\rfloor$ and $b_j=\lfloor\frac{b(j)}{r}\rfloor$. Note that $a_{(c-1)/2} \geq \cdots \geq a_{c} = 0$, $0 = b_{(c-1)/2} \leq \cdots \leq b_{c}$, $a_{(c-1)/2} > b_{(c-1)/2}$ and $a_{c} < b_{c}$. Let $w$ be the unique integer in $\{\frac{c-1}{2},\ldots,c-1\}$ such that $a_{w} > b_{w}$ and $a_{w+1} \leq b_{w+1}$. By applying Lemma~\ref{l:altConstruction} with $u=w+1$ (or trivially if $w=c-1$) we have $\SP(n,k) \geq a_{w+1}$. By definition of $a$, we have $a(w+1)=a(w)-2\binom{n/2}{w+1}\binom{n/2}{c-w-1}$ and hence
\begin{equation}\label{e:jumpSizeAlt}
\SP(n,k) \geq a_{w+1} = \left\lfloor\frac{a(w)-2\binom{n/2}{w+1}\binom{n/2}{c-w-1}}{k-r}\right\rfloor \geq a_{w}-\frac{2\binom{n/2}{w+1}\binom{n/2}{c-w-1}}{k-r}-1.
\end{equation}

We now show that
\begin{equation}\label{e:coverageAlt}
(c+1)b(w)=(n-c)\left(\mbinom{n}{c}-a(w)\right)-\delta'(n-2w)\mbinom{n/2}{w}\mbinom{n/2}{c-w},
\end{equation}
where $\delta'=1$ if $w \geq \frac{c+1}{2}$ and $\delta'=0$ if $w = \frac{c-1}{2}$. We may assume $w \geq \frac{c+1}{2}$, for otherwise $w=\frac{c-1}{2}$, $b(w)=0$, $a(w)=\binom{n}{c}$ and \eqref{e:coverageAlt} holds. Consider applying Lemma~\ref{l:altConstruction} with $u=w$, let $\A$ and $\B$ be as defined in its proof, and let $\A^c=\binom{X}{c} \setminus \A$. Note that $|\A|=a(w)$, $|\B|=b(w)$ and $|\A^c|=\binom{n}{c}-a(w)$. We now count, in two ways, the number of pairs $(S,B)$ such that $S \in \A^c$, $B \in \B$ and $S \subseteq B$.
\begin{itemize}
    \item
Each of the $b(w)$ sets in $\B$ has exactly $c+1$ subsets in $\binom{X}{c}$ and each of these is in $\A^c$, because no set in $\A$ is a subset of a set in $\B$.
    \item
By the definition of $\mathcal{A}$, $\max(|S \cap X_1|,|S \cap X_2|) \leq w$ for each $S \in \A^c$. Each of the $\binom{n}{c}-a(w)$ sets in $\A^c$ has $n-c$ supersets in $\binom{X}{c+1}$.  For each $S \in \A^c$ such that $\max(|S \cap X_1|,|S \cap X_2|) \leq w-1$, all of these supersets of $S$ are in $\B$. For each of the $2\binom{n/2}{w}\binom{n/2}{c-w}$ sets $S \in \A^c$ such that $\max(|S \cap X_1|,|S \cap X_2|) = w$, exactly $\frac{n}{2}-w$ of these supersets of $S$ are not in $\B$.
\end{itemize}
By equating our two counts, \eqref{e:coverageAlt} holds.

Using \eqref{e:jumpSizeAlt} and \eqref{e:coverageAlt} in place of \eqref{e:jumpSize} and \eqref{e:coverage}, it is now routine to obtain the desired conclusion by following the argument from the case of the proof where $n$ and $ck$ are even.

\smallskip\noindent\textbf{Case 2.} Suppose that $n$ is odd, or that $n$ is even and $r=1$. By our hypotheses, $k-r \rightarrow \infty$. Note that
\[\frac{\MMS(n-1,k)}{\MMS(n,k)} = \frac{k-r+\frac{r(c+1)}{n-c}}{k-r+1+\frac{(r-1)(c+1)}{n-c-1}}\cdot\mfrac{n-c}{n} \geq  \mfrac{k-r}{k-r+1}\cdot\mfrac{k-1}{k} = 1-o(1),\]
where the first inequality follows because $\frac{r(c+1)}{n-c} \geq \frac{(r-1)(c+1)}{n-c-1}$ and $\frac{n-c}{n} \geq \frac{k-1}{k}$ and the second equality follows because $k-r \rightarrow \infty$. Hence we have $\MMS(n-1,k)=\MMS(n,k)(1-o(1))$. Thus, if $\SP(n-1,k)=\MMS(n-1,k)(1-o(1))$, we have
\begin{equation}\label{e:monotonicityChain}
\SP(n,k) \geq \SP(n-1,k)=\MMS(n-1,k)(1-o(1))=\MMS(n,k)(1-o(1)).
\end{equation}
If $n$ is even and $r=1$, we have $\SP(n-1,k)=\MMS(n-1,k)$ from \cite{MeaMouSte} and thus \eqref{e:monotonicityChain} definitely holds. So the theorem holds in all the cases where $n$ is even. But, having established this, we may assume that $n$ is odd and we know that $\SP(n-1,k)=\MMS(n-1,k)(1-o(1))$. Hence the proof is complete, using \eqref{e:monotonicityChain}.
\end{proof}

\section{Proof of Theorem~\ref{t:upperBound}}\label{s:upperBoundProof}

Let $X$ be a ground set, let $\S$ be a family of subsets of $X$, and let $i$ be an integer. If each set in $\S$ has cardinality at least $i$, then we define $\Delta^i(\S)$ to be the family of all sets in $\binom{X}{i}$ that are subsets of some set in $\S$. Similarly, if each set in $\S$ has cardinality at most $i$, then we define the $\nabla^i(\S)$ to be the family of all sets in $\binom{X}{i}$ that are supersets of some set in $\S$.

The following theorem, due to Lov\'{a}sz \cite[p. 95]{Lov}, gives a convenient approximation to the Kruskal-Katona theorem (see \cite{Kat,Kru} for the original theorem).

\begin{theorem}[\cite{Lov}]\label{t:LovaszKK}
If $i \geq 2$ is an integer, $X$ is a set and $\S \subseteq \binom{X}{i}$, then
$|\Delta^{i-1}(\S)| \geq \L_i(|\S|)$.
\end{theorem}

Recall that the function $\L_i$ was defined just prior to the statement of Theorem~\ref{t:upperBound} in the introduction. It will be important for our purposes that, for a fixed integer $i \geq 2$, $\L_i(x)$ is monotonically increasing and concave in $x$ for $x \geq 1$ (see \cite[Lemma 4]{Buk}). We will make use of the following simple consequence of Theorem~\ref{t:LovaszKK}.

\begin{lemma}\label{l:shadowSize}
Let $H$ be a clutter with edge set $\E$, and $c$ be a positive integer such that $|E| \geq c$ for each $E \in \E$. Then $|\Delta^{c}(\E)| \geq \min(|\E|,\binom{2c+1}{c}+1)$.
\end{lemma}

\begin{proof}
If each edge in $\E$ has cardinality $c$, then $\Delta^{c}(\E)=\E$ and the result holds trivially. So we may suppose inductively that the maximum cardinality of an edge in $\E$ is $j \geq c+1$ and that the result holds if the maximum cardinality of an edge in $\E$ is $j-1$.

Let $\E_i=\{E \in \E:|E|=i\}$ for each $i \in \{c,\ldots,j\}$, and let $H^*$ be a hypergraph with vertex set $V(H)$ and edge set $\E^*=(\E \setminus \E_j) \cup \Delta^{j-1}(\E_j)$. Because $H$ is a clutter, $H^*$ is a clutter and $\Delta^{j-1}(\E_j)$ is disjoint from $\E_{j-1}$.
\begin{itemize}
    \item
If $|\E_j| \leq \binom{2j-1}{j}$, then $|\E_j|=\binom{y}{j}$ for some real $y \leq 2j-1$ and hence $|\Delta^{j-1}(\E_j)| \geq \binom{y}{j-1} \geq \binom{y}{j} = |\E_j|$ using Theorem~\ref{t:LovaszKK}. Thus $|\E^*| \geq |\E|$.
    \item
If $|\E_j| > \binom{2j-1}{j}$, then $|\Delta^{j-1}(\E_j)| > \binom{2j-1}{j-1} \geq \binom{2c+1}{c}$ by Theorem~\ref{t:LovaszKK} and so $|\E^*| \geq \binom{2c+1}{c}+1$.
\end{itemize}
So in either case $|\E^*| \geq \min(|\E|,\binom{2c+1}{c}+1)$. The result now follows by applying our inductive hypothesis to $\E^*$ and noting that $\Delta^{c}(\E^*)=\Delta^{c}(\E)$.
\end{proof}

The bulk of the work of proving Theorem~\ref{t:upperBound} is accomplished in the following lemma. It establishes that \eqref{e:upperBound} holds subject to the existence of a clutter with desirable properties. It then only remains to show that, given an $(n,k)$-Sperner partition system with $p$ partitions, a clutter satisfying the hypotheses of Lemma~\ref{l:localisedClutter} can be obtained by considering the partition classes containing a particular element. (In fact, we must also do some tedious checking to ensure that \eqref{e:upperBound} holds for ``small'' values of $p$ not covered by Lemma~\ref{l:localisedClutter}.) In the proof of Theorem~\ref{t:upperBound}, this special element is chosen as one that, according to a certain metric, tends overall to appear in smaller partition classes.

\begin{lemma}\label{l:localisedClutter}
Let $n$ and $k$ be integers such that $n \geq 2k+2$ and $k \geq 3$, and let $p$ be an integer such that
\[p \geq \max\left(\tfrac{2n}{c(2k-r)}\left(\tbinom{n-1}{c-1} - \tbinom{2c+1}{c-1}\right),\tbinom{n-1}{c-1}+1\right).\]
If there is a clutter with $n-1$ vertices and edge set $\E$ such that $|\E| \geq p$ and $\sum_{E \in \E}\frac{c-|E|}{|E|+1} \geq \frac{p(k-r)}{n}$, then
\[\big\lceil \big(1-\tfrac{r(c+1)}{n}\big)p \big\rceil+\L_c\left(\big\lfloor \tfrac{r(c+1)}{n}p \big\rfloor\right) \leq \mbinom{n-1}{c-1}.\]
\end{lemma}

\begin{proof}
Let $H$ be a clutter satisfying the hypotheses of the lemma and let $X'=V(H)$. For each $i \in \{0,\ldots,n-1\}$, let $\E_i=\{E \in \E: |E|=i\}$ and let $\E_{>c}=\E_{c+1} \cup \cdots \cup \E_{n-1}$. We abbreviate $\lceil\tfrac{pc(k-r)}{n} \rceil$ to $a_0$. Note that $a_0=\lceil (1-\tfrac{r(c+1)}{n})p\rceil$ using $n=ck+r$. We consider two cases according to minimum cardinality of an edge in $\E$.

\noindent{\bf Case 1.} Suppose that $|E| \geq c-1$ for each $E \in \E$. Then the only edges in $\E$ that make a positive contribution toward $\sum_{E \in \E}\frac{c-|E|}{|E|+1}$ are those in $\E_{c-1}$ and so by our hypotheses we must have $\frac{1}{c}|\E_{c-1}| \geq \frac{p(k-r)}{n}$ and hence $|\E_{c-1}| \geq a_0$. Also, because $p \geq \binom{n-1}{c-1}+1$, we have $\E \nsubseteq \binom{X'}{c-1}$ and hence $|\E_c|+|\E_{>c}| \geq 1$.  Let $H^*$ be the hypergraph with vertex set $X'$ and edge set $\E^*=(\E \setminus \E_{>c}) \cup \Delta^c(\E_{>c})$. Then $\E^*=\E_{c-1} \cup \E_c \cup \Delta^c(\E_{>c})$ and, because $H$ is a clutter, $H^*$ is a clutter and $\Delta^c(\E_{>c})$ is disjoint from $\E_c$.

There are $\binom{n-1}{c-1}$ sets in $\binom{X'}{c-1}$, and because $H^*$ is a clutter each of these can be in at most one of $\E_{c-1}$ and $\Delta^{c-1}(\E_c \cup \Delta^c(\E_{>c}))$. Thus, by Theorem~\ref{t:LovaszKK},
\begin{equation}\label{e:KK}
|\E_{c-1}| + \L_{c}\big(|\E_c|+|\Delta^c(\E_{>c})|\big) \leq \mbinom{n-1}{c-1}.
\end{equation}
We consider two subcases according to the value of $|\E_{>c}|$.

\noindent{\bf Case 1a.} Suppose that $|\E_{>c}| \leq \binom{2c+1}{c}$. Then $|\Delta^c(\E_{>c})| \geq |\E_{>c}| = |\E|-|\E_c|-|\E_{c-1}|$ by Lemma~\ref{l:shadowSize}. So $|\E_c|+|\Delta^c(\E_{>c})| \geq \max(p-|\E_{c-1}|,1)$ because $|\E| \geq p$ and $|\E_c|+|\E_{>c}| \geq 1$. Thus, using the fact that $\L_{c}$ is monotonically increasing, \eqref{e:KK} implies that
\[
f(|\E_{c-1}|) \leq \mbinom{n-1}{c-1} \mbox{ where } f(a) = a+\L_c\left(\max\left(p-a,1\right)\right).
\]

Consider $f$ as a function on the real domain $a_0 \leq a \leq  |\E|$, noting that we have seen $a_0 \leq |\E_{c-1}| \leq |\E|$. Because $f(|\E_{c-1}|) \leq \binom{n-1}{c-1}$, certainly the global minimum of $f$ is at most $\binom{n-1}{c-1}$. Now, $f$ is monotonically increasing for $p-1 < a \leq |\E|$ and, because $\L_c$ is concave, $f$ is concave for $a_0 \leq a \leq  p-1$. Thus, $f$ achieves its global minimum either at $a=a_0$ or at $a=p-1$. However, $f(p-1)=p-1+c$ and $p-1 \geq \binom{n-1}{c-1}$ by our hypotheses. Thus $f$ achieves its global minimum at $a=a_0$ and we have $f(a_0) \leq \binom{n-1}{c-1}$. Now the result follows because
\[f(a_0)=a_0+\L_c(p-a_0) = \big\lceil \big(1-\tfrac{r(c+1)}{n}\big)p \big\rceil+\L_c\left(\big\lfloor \tfrac{r(c+1)}{n}p \big\rfloor\right).\]

\smallskip\noindent{\bf Case 1b.} Suppose that $|\E_{>c}| > \binom{2c+1}{c}$. Because $\sum_{E \in \E}\frac{c-|E|}{|E|+1} \geq \frac{p(k-r)}{n}$,
\[\tfrac{p(k-r)}{n} \leq \tfrac{1}{c}|\E_{c-1}|-\medop\sum_{i = c+1}^{n-1}\tfrac{i-c}{i+1}|\E_i| \leq \tfrac{1}{c}|\E_{c-1}|-\tfrac{1}{c+2}|\E_{>c}|,\]
where the last inequality follows because $\frac{i-c}{i+1} \geq \tfrac{1}{c+2}$ for each $i \in \{c+1,\ldots,n-1\}$. Thus $|\E_{>c}| \leq \frac{c+2}{c}|\E_{c-1}|-\tfrac{p(c+2)(k-r)}{n}$. Also, $|\Delta^c(\E_{>c})| > \binom{2c+1}{c}$ by the hypothesis of this subcase and Lemma~\ref{l:shadowSize}. Combining these facts and $|\E| \geq p$, we have
\[|\E_c|+|\Delta^c(\E_{>c})| = |\E|-|\E_{c-1}|-|\E_{>c}|+|\Delta^c(\E_{>c})| > \max\left(\tfrac{p(c+1)(2k-r)}{n}-\tfrac{2c+2}{c}|\E_{c-1}|,0\right)+\tbinom{2c+1}{c}.\]
Thus, \eqref{e:KK} implies that
\[
g(|\E_{c-1}|) < \mbinom{n-1}{c-1} \mbox{ where } g(a) = a+\L_c\left(\max\left( \tfrac{p(c+1)(2k-r)}{n}-\tfrac{2c+2}{c}a,0\right)+\tbinom{2c+1}{c}\right).
\]

Consider $g$ as function on the real domain $a_0 \leq a \leq  |\E|$ and note that the global minimum of $g$ is less than $\binom{n-1}{c-1}$. Now, $g$ is monotonically increasing for $a_1 < a \leq |\E|$ and concave for $a_0 \leq a \leq  a_1$, where $a_1=\frac{pc(2k-r)}{2n}$. Thus, it achieves its global minimum either at $a=a_0$ or at $a=a_1$. However, $g(a_1)=a_1+\binom{2c+1}{c-1} \geq \binom{n-1}{c-1}$ using the hypothesis that $p \geq \frac{2n}{c(2k-r)}(\binom{n-1}{c-1} - \binom{2c+1}{c-1})$. Thus we have $g(a_0) < \binom{n-1}{c-1}$. Now, setting $\delta=a_0-\frac{pc(k-r)}{n}$ and noting that $0 \leq \delta <1$,
\[g(a_0)=a_0+\L_c\left(p-a_0+\tbinom{2c+1}{c}-\tfrac{c+2}{c}\delta\right) \geq a_0+\L_c(p-a_0).\]
As in Case 1a, the result follows.

\smallskip\noindent{\bf Case 2.} Suppose that $|E| \leq c-2$ for some $E \in \E$. Using Case 1 as a base case, we may suppose inductively that the minimum cardinality of an edge in $\E$ is $j \leq c-2$ and that the lemma holds when the minimum cardinality of an edge in $\E$ is $j+1$. For any family $\S$ of subsets of $X$, define $d'(\S)=\sum_{S \in \S}\frac{c-|S|}{|S|+1}$. Note we have assumed that $d'(\E) \geq \frac{p(k-r)}{n}$.

Let $H^*$ be the hypergraph with vertex set $X'$ and edge set $\E^*=(\E \setminus \E_{j}) \cup \nabla^{j+1}(\E_j)$. Because $H$ is a clutter, $H^*$ is a clutter and $\nabla^{j+1}(\E_j)$ is disjoint from $\E_{j+1}$. Thus it suffices to show that $d'(\nabla^{j+1}(\E_j)) \geq d'(\E_{j})$ and $|\nabla^{j+1}(\E_j)| \geq |\E_{j}|$ because then we will be able to apply our inductive hypothesis to $H^*$ to obtain the required result.

Each edge in $\E_{j}$ is a subset of $n-j-1$ edges in $\nabla^{j+1}(\E_j)$, and each edge in $\nabla^{j+1}(\E_j)$ is a superset of at most $j+1$ edges in $\E_{j}$. Thus $|\nabla^{j+1}(\E_j)| \geq \frac{n-j-1}{j+1} |\E_{j}|$ and
\[d'(\nabla^{j+1}(\E_j))=\tfrac{c-j-1}{j+2}|\nabla^{j+1}(\E_j)| \geq \tfrac{(c-j-1)(n-j-1)}{(j+1)(j+2)}|\E_{j}|=\tfrac{(c-j-1)(n-j-1)}{(c-j)(j+2)}d'(\E_{j}),\]
where the second equality follows because $d'(\mathcal{E}_{j})=\frac{c-j}{j+1}|\mathcal{E}_{j}|$. Thus $d'(\nabla^{j+1}(\E_j)) \geq d'(\E_{j})$ and $|\E^*| \geq |\E|$ as required because, using $j \in \{0,\ldots,c-2\}$ and $k \geq 3$, we have $c-j-1 \geq \frac{1}{2}(c-j)$ and $n-j-1 \geq 2(j+2)$.
\end{proof}

\begin{proof}[\textup{\textbf{Proof of Theorem~\ref{t:upperBound}}}]
Let $p_0=\SP(n,k)$, let $X$ be a set with $|X|=n$, and let $\P$ be an $(n,k)$-Sperner partition system on ground set $X$ with $p_0$ partitions. We may assume $r \neq 0$ because, when $r=0$, $\SP(n,k)=\binom{n-1}{c-1}$ and \eqref{e:upperBound} clearly holds with $p=\binom{n-1}{c-1}$. So, in addition to $n \geq 2k+2$, we have $n \geq 4c+1$. Let $p_1 = \max(\tfrac{2n}{c(2k-r)}(\tbinom{n-1}{c-1} - \tbinom{2c+1}{c-1}),\tbinom{n-1}{c-1}+1)$.

\smallskip\noindent{\bf Case 1.} Suppose that $p_0 \geq p_1$. We will find a clutter satisfying the conditions of Lemma~\ref{l:localisedClutter} and so complete the proof. For each $x \in X$, let $\P(x)$ be the set of all partition classes of $\P$ that contain $x$. For a subset $S$ of $X$ we define $d(S)=c+1-|S|$, and for a family $\S$ of subsets of $X$ we define $d(\S)=\sum_{S \in \S}d(S)$. Note that, for each partition $\pi$ in $\P$, we have $d(\pi)=k-r$ because $\pi$ has exactly $k$ classes and the sum of the cardinalities of the classes is equal to $n=ck+r$. For a vertex $x \in X$, we further define $d(x)=\sum_{S \in \P(x)} \frac{d(S)}{|S|}$. Thus we have that $\sum_{x \in X}d(x)=\sum_{\pi \in \P}d(\pi)=p_0(k-r)$. Let $z$ be an element of $X$ such that $d(z) \geq d(x)$ for each $x \in X$ and observe that $d(z) \geq \frac{p_0(k-r)}{n}$. Let $H$ be the hypergraph with vertex set $X'=X \setminus \{z\}$ and edge set $\E=\{S \setminus \{z\}: S \in \P(z)\}$. Note that $H$ is a clutter and $|\E|=p_0$ because $\P$ is a Sperner partition system with $p_0$ partitions. Thus, because $d(z) \geq \frac{p_0(k-r)}{n}$ and $p_0 \geq p_1$, $H$ satisfies the conditions of Lemma~\ref{l:localisedClutter} and we can apply it to produce the required result.

\smallskip\noindent{\bf Case 2.} Suppose that $p_0 < p_1$. In this case we show directly that \eqref{e:upperBound} holds for some real number $p \geq p_1$ and hence that it holds for $p=p_0$ (recall that the left hand side of \eqref{e:upperBound} is nondecreasing in $p$).

\smallskip\noindent{\bf Case 2a.} Suppose that $c=2$. Then, when $r = 2$, we have $p_1=2k+2$ and \eqref{e:upperBound} holds because $\L_2(6) \leq 5$. When $r \geq 3$, $p_1=\frac{2k+r}{2k-r}(2k+r-6)$ and it can be seen that \eqref{e:upperBound} holds if and only if $\L_2(\lfloor\frac{3r(2k+r-6)}{2k-r}\rfloor) \leq \lfloor r+5+\frac{2r(r-3)}{2k-r} \rfloor$. This holds for each integer $r \geq 4$ because then $\frac{3r(2k+r-6)}{2k-r}<9r \leq \binom{r+5}{2}$. It also holds for $r=3$ because $\L_2(9)\leq 8$.

\smallskip\noindent{\bf Case 2b.} Suppose that $c \geq 3$. Let $p_2= \tfrac{2n}{c(2k-r)}(\tbinom{n-1}{c-1}-1)$ and note that $p_2 \geq p_1$. Noting that $1-\tfrac{r(c+1)}{n}=\frac{c(k-r)}{n}$, it can be seen that \eqref{e:upperBound} will hold with $p=p_2$ provided that
\begin{equation}\label{e:upperBoundProofApprox}
\L_c\left(\tfrac{2r(c+1)}{c(2k-r)}\left(\tbinom{n-1}{c-1}-1\right) \right) \leq \left\lfloor\tfrac{r}{2k-r}\tbinom{n-1}{c-1}+\tfrac{2k-2r}{2k-r}\right\rfloor.
\end{equation}
Let $z=\tfrac{2r(c+1)}{c(2k-r)}(\tbinom{n-1}{c-1}-1)$ be the argument of $\L_c$ in \eqref{e:upperBoundProofApprox} and note that if $z \geq \binom{3c+1}{c}$, then it follows from the definition of $\L_c$ that $\L_c(z) \leq \frac{c}{2c+2}z$ and thus that \eqref{e:upperBoundProofApprox} holds. Because $r \geq 1$, we have $z \geq \frac{2c+2}{2n-c-2}(\binom{n-1}{c-1}-1)$. This latter expression is an increasing function of $n$ for $n \geq 4c+1$. Thus, for $c \geq 9$ we have $z \geq \binom{3c+1}{c}$ because $z \geq \frac{2c+2}{2n-c-2}(\binom{n-1}{c-1}-1) \geq \frac{2c+2}{7c}(\binom{4c}{c-1}-1)$ and
\[\tfrac{2c+2}{7c}\left(\tbinom{4c}{c-1}-1\right)\big/\tbinom{3c+1}{c} = \tfrac{2c+2}{21c+7}\tbinom{4c}{c-1}\big/\tbinom{3c}{c-1}-\tfrac{2c+2}{7c}\big/\tbinom{3c+1}{c} \geq  \tfrac{2c+2}{21c+7}(\tfrac{4}{3})^{c-1}-10^{-7} \geq 1.\]
Furthermore, by explicit calculation, we have $z \geq \frac{2c+2}{7c}(\binom{4c}{c-1}-1) \geq \binom{3c+1}{c}$ for $c=8$. We also have $z \geq \frac{2c+2}{2n-c-2}(\binom{n-1}{c-1}-1) \geq \binom{3c+1}{c}$ for $c\in \{4,5,6,7\}$ and $n \geq 31$ and for $c=3$ and $n \geq 61$. This leaves only a limited number of pairs $(n,k)$ to be checked. Using a computer, it is routine to compute $p_1$ for each pair and verify that \eqref{e:upperBound} holds for $p=p_1$.
\end{proof}

We conclude this section by showing that a slightly weaker version of the upper bound implied by Theorem~\ref{t:upperBound} can be written in a form that is very reminiscent of the expression for $\MMS(n,k)$, and that this implies that our upper bound is always at least as good as $\MMS(n,k)$.

\begin{corollary}\label{c:MMSFormUpperBound}
If $n$ and $k$ are integers such that $n \geq 2k+2$, $k \geq 4$ and $r \neq 0$,
\begin{equation}\label{e:MMSFormUpperBound}
\SP(n,k) \leq \frac{\binom{n}{c}}{(k-r)+\frac{r(c+1)}{q-c+1}}
\end{equation}
where $q$ is the real number such that $q \geq c$ and $\binom{q}{c}=\frac{r(c+1)}{n}\cdot\SP(n,k)$. Furthermore, the bound implied by Theorem~$\ref{t:upperBound}$ is less than $\MMS(n,k)$.
\end{corollary}

\begin{proof}
Observe that $\SP(n,k)\geq \NLB(n,k) = \frac{1}{k}\binom{ck}{c}$ implies $\tfrac{r(c+1)}{n}\cdot\SP(n,k) > \tfrac{r}{k^2}\binom{ck}{c}$. Further, it is routine to verify $\tfrac{r}{k^2}\binom{ck}{c}\geq\tfrac{r}{16}\binom{4c}{c}>\binom{2c-1}{c}$ since $r \geq 1$ when $c \geq 3$ and $r \geq 2$ when $c = 2$. Thus we have $\tfrac{r(c+1)}{n}\cdot\SP(n,k) > \binom{2c-1}{c}$.
It follows that $q$ is well defined. Further, because $\L_c(1)=c$, $\L_c(\binom{2c-1}{c})=\binom{2c-1}{c}$ and $\L_c$ is concave, the derivative of $\L_c(x)$ is less than $1$ for all $x \geq \binom{2c-1}{c}$ and hence, for any real $\epsilon > 0$,
\[\L_c\left(\big\lfloor\tfrac{r(c+1)}{n}\cdot\SP(n,k)\big\rfloor+\epsilon\right) < \L_c\left(\big\lfloor\tfrac{r(c+1)}{n}\cdot\SP(n,k)\big\rfloor\right)+\epsilon.\]
Thus we can deduce from Theorem~$\ref{t:upperBound}$ the slightly weaker conclusion that
\begin{equation}\label{e:weakForm}
\left(1-\tfrac{r(c+1)}{n}\right) \cdot \SP(n,k) +\L_c\left( \tfrac{r(c+1)}{n} \cdot \SP(n,k) \right) \leq \mbinom{n-1}{c-1}.
\end{equation}
By applying $\L_c(x) = \frac{c}{q-c+1}x$ in \eqref{e:weakForm} and solving for $\SP(n,k)$ we obtain \eqref{e:MMSFormUpperBound}.

Now, using $\SP(n,k) \leq \MMS(n,k)$, we have
\[\tfrac{r(c+1)}{n}\cdot\SP(n,k) \leq \frac{\binom{n-1}{c} }{1+\frac{(n-c)(k-r)}{r(c+1)}} \leq \mfrac{1}{2}\mbinom{n-1}{c}.\]
Thus, $q<n-1$ and so the bound implied by this corollary, and hence the bound implied by Theorem~\ref{t:upperBound}, is less than $\MMS(n,k)$.
\end{proof}

\section{The case \texorpdfstring{$\bm{n = 3k-6}$}{n=3k-6}}\label{s:infFamily}

In this section we exhibit a new infinite family of parameter sets $(n,k)$ for which we can precisely determine $\SP(n,k)$. For this family, the value of $\SP(n,k)$ matches the upper bound given by Theorem~\ref{t:upperBound}, and hence it supplies examples both of the theorem's usefulness and of situations in which its bound is tight.

\begin{lemma}\label{l:family}
Let $k \geq 11$ be an integer such that $k \not\equiv 4 \mod{6}$ and let $n=3k-6$. Then $\SP(n,k)=\lfloor\frac{1}{2}(k-2)^2\rfloor$.
\end{lemma}

\begin{proof}
First, suppose for a contradiction that $\SP(n,k) \geq \lfloor\frac{1}{2}(k-2)^2\rfloor+1$. Note that $\lfloor\frac{1}{2}(k-2)^2\rfloor+1=\frac{1}{2}(k-2)^2+\delta$ where $\delta=\frac{1}{2}$ if $k$ is odd and $\delta=1$ if $k$ is even. Then Theorem~\ref{t:upperBound} implies that \eqref{e:upperBound} holds with $n=3k-6$ and $p=\frac{1}{2}(k-2)^2+\delta$ and hence, via routine calculation,
\[2k-3+\L_2(\lfloor\tfrac{1}{2}(k^2-8k+12)\rfloor) \leq 3k-7.\]
However, because $\binom{k-4}{2}=\frac{1}{2}(k^2-9k+20)<\lfloor\tfrac{1}{2}(k^2-8k+12)\rfloor$ for $k \geq 11$, we have that $\L_2(\lfloor\tfrac{1}{2}(k^2-8k+12)\rfloor)>k-4$ and hence a contradiction.

Now we construct an $(n,k)$-Sperner partition system with $\lfloor\frac{1}{2}(k-2)^2\rfloor$ partitions and so complete the proof. Let $p=\lfloor\frac{1}{2}(k-2)^2\rfloor$, let $X_1$, $X_2$ and $X_3$ be disjoint sets such that $|X_1|=|X_2|=|X_3|=k-2$, and let $X=X_1 \cup X_2 \cup X_3$. For each $i \in \{1,2,3\}$, let
\begin{align*}
\mathcal{A}_i&=\{A \subseteq X: \mbox{$|A|=2$ and $|A \cap X_j|=1$ for each $j \in \{1,2,3\}\setminus\{i\}$}\} \\
\mathcal{B}_i&=\{B \subseteq X: \mbox{$|B|=3$ and $B \subseteq X_i$}\}.
\end{align*}
Let $\mathcal{A}=\mathcal{A}_1 \cup \mathcal{A}_2 \cup \mathcal{A}_3$ and $\mathcal{B}=\mathcal{B}_1 \cup \mathcal{B}_2 \cup \mathcal{B}_3$, and let $H$ be the hypergraph with vertex set $X$ and edge set $\mathcal{A} \cup \mathcal{B}$. Note that no set in $\mathcal{A}$ is a subset of a set in $\mathcal{B}$ and thus $H$ is a clutter. Observe that, for each $i \in \{1,2,3\}$, any permutation of $X_i$ is an automorphism of $H$. Let $C$ be a set of $p$ colours other than black.  By Lemma~\ref{l:colouringToSystem}, it suffices to find an edge colouring $\gamma$ of $H$ with colour set $C \cup \{\hbox{black}\}$ such that, for each $c \in C$, colour $c$ is assigned to $6$ edges in $\mathcal{A}$ and $k-6$ edges in $\B$ and $\sum_{x \in X_i}\deg^\gamma_c(x)=k-2$ for each $i \in \{1,2,3\}$.

We now describe how to find an edge colouring that satisfies the conditions we have specified. If $k \equiv 1 \mbox{ or } 2 \mod{6}$, then $p \equiv 0 \mod{3}$ and we let $\{C_1,C_2,C_3\}$ be a partition of $C$ such that $|C_1|=|C_2|=|C_3|=\frac{p}{3}$. If $k \equiv 5 \mod{6}$, then $p \equiv 1 \mod{3}$ and we let $\{C_1,C_2,C_3\}$ be a partition of $C$ such that $|C_1|=\frac{p+2}{3}$ and  $|C_2|=|C_3|=\frac{p-1}{3}$.  We describe how to choose the edges from $\A$ in each non-black colour class of $\gamma$; the remaining edges in each non-black class can be chosen from $\B$ arbitrarily subject to our specified conditions, and then any remaining edges are coloured black.
\begin{itemize}
    \item
If $k\equiv 0 \mod{3}$ then, for each $c \in C$, assign colour $c$ to two edges in $\mathcal{A}_i$ for each $i \in \{1,2,3\}$;
    \item
If $k\equiv 1 \mod{6}$ then, for each $j \in \{1,2,3\}$ and $c \in C_j$, assign colour $c$ to four edges in $\mathcal{A}_j$ and one edge in $\mathcal{A}_i$ for each $i \in \{1,2,3\} \setminus \{j\}$;
    \item
If $k\equiv 2 \mod{3}$ then, for each $j \in \{1,2,3\}$ and $c \in C_j$, assign colour $c$ to three edges in $\mathcal{A}_i$ for each $i \in \{1,2,3\} \setminus \{j\}$.
\end{itemize}
It only remains to check that there are sufficiently many edges in $\A_i$ and $\B_i$ for each $i \in \{1,2,3\}$ that we can choose an edge colouring in this manner. Using the fact that $|\mathcal{A}_i|=(k-2)^2$ and $|\mathcal{B}_i|=\binom{k-2}{3}$ for each $i \in \{1,2,3\}$, it is routine to check this by considering cases according to the congruence class of $k$ modulo 6.
\end{proof}

Again, the Sperner partition systems constructed to prove Lemma~\ref{l:family} are almost uniform.

\section{Bounds for small \texorpdfstring{$\bm{n}$}{n} and \texorpdfstring{$\bm{k}$}{k}}\label{s:smallParams}

We conclude this paper by displaying the values of the upper and lower bounds we have obtained for some small parameters $(n,k)$.

In Table~\ref{tab:smallBounds} we list, for $4 \leq k \leq 7$ and $2k+2 \leq n \leq 33$ a lower bound and an upper bound on $\SP(n,k)$ in the top and bottom rows respectively of the appropriate cell. The upper bound is the bound implied by Theorem~\ref{t:upperBound} and is followed by the improvement over $\MMS(n,k)$ in brackets. The lower bound is the best one attainable via our results and those of \cite{LiMea,MeaMouSte} and is followed by the source of the bound according to the following key. ``M'' refers to a bound obtained through the monotonicity of $\SP(n,k)$ in $n$; ``\cite{LiMea}'' refers to one of the bounds given in \cite{LiMea} (and stated in our introduction); ``L\ref{l:productConstruction}'' refers to Lemma~\ref{l:productConstruction} and is followed by the values of $m$ and $n$ used; and finally ``L\ref{l:mainConstruction}'' and ``L\ref{l:altConstruction}'' refer to Lemmas~\ref{l:mainConstruction} and \ref{l:altConstruction} and are followed by the value of $u$ used. The exception to the above is when $k$ divides $n$, in which case the known exact value of $\SP(n,k)$ is placed by itself in the cell.

Figures~\ref{fig:k5} and \ref{fig:k10} visualise bounds on $\SP(n,k)$ for the example values $k=5$ and $k=10$ respectively. Values of $n$ between $2k+2$ and $100$ appear on the horizontal axis, and above each are a grey and a black line segment. The grey segment gives the interval between $\NLB(n,k)$ and $\MMS(n,k)$, whereas the black segment gives the interval between the best known lower and upper bounds on $\SP(n,k)$ according to the results in this paper and in \cite{LiMea,MeaMouSte}. Note that the vertical axis is log scaled.

\begin{table}[H]
\centering
\small
\setstretch{1.0}
\begin{tabular}{|r|rr|rr|rr|rr|}
\hline
\multicolumn{1}{|c|}{$n$} & \multicolumn{2}{c|}{$k=4$}          & \multicolumn{2}{c|}{$k=5$}           & \multicolumn{2}{c|}{$k=6$}                             & \multicolumn{2}{c|}{$k=7$}                            \\ \hline
\multirow{2}{*}{10}       & 10                       & L10 (1)  &                         &            &                        &                               &                       &                               \\
                          & 11                       & (5)      &                         &            &                        &                               &                       &                               \\ \hline
\multirow{2}{*}{11}       & 11                       & \cite{LiMea}        &                         &            &                        &                               &                       &                               \\
                          & 19                       & (8)      &                         &            &                        &                               &                       &                               \\ \hline
\multirow{2}{*}{12}       & \multirow{2}{*}{55}      &          & 12                      & L10 (1)    &                        &                               &                       &                               \\
                          &                          &          & 13                      & (5)        &                        &                               &                       &                               \\ \hline
\multirow{2}{*}{13}       & 55                       & M        & 12                      & M          &                        &                               &                       &                               \\
                          & 72                       & (12)     & 19                      & (8)        &                        &                               &                       &                               \\ \hline
\multirow{2}{*}{14}       & 55                       & M        & 17                      & L10 (1)    & 13                     & \cite{LiMea} &                       &                               \\
                          & 110                      & (23)     & 33                      & (12)       & 15                     & (5)                           &                       &                               \\ \hline
\multirow{2}{*}{15}       & 55                       & M        & \multirow{2}{*}{91}     &            & 13                     & M                             &                       &                               \\
                          & 190                      & (37)     &                         &            & 20                     & (8)                           &                       &                               \\ \hline
\multirow{2}{*}{16}       & \multirow{2}{*}{455}     &          & 91                      & M          & 28                     & L10 (1)                       & 15                    & \cite{LiMea} \\
                          &                          &          & 114                     & (16)       & 29                     & (13)                          & 17                    & (5)                           \\ \hline
\multirow{2}{*}{17}       & 455                      & M        & 91                      & M          & 28                     & M                             & 15                    & M                             \\
                          & 636                      & (67)     & 162                     & (28)       & 51                     & (17)                          & 21                    & (8)                          \\ \hline
\multirow{2}{*}{18}       & 648                      & L10 (2)  & 91                      & M          & \multirow{2}{*}{136}   &                               & 27                    & L10 (1)                       \\
                          & 994                      & (133)    & 243                     & (48)       &                        &                               & 30                    & (10)                          \\ \hline
\multirow{2}{*}{19}       & 648                      & M        & 91                      & M          & 136                    & M                             & 27                    & M                             \\
                          & 1\,719                     & (219)    & 410                     & (74)       & 167                    & (17)                          & 42                    & (17)                          \\ \hline
\multirow{2}{*}{20}       & \multirow{2}{*}{3\,876}    &          & \multirow{2}{*}{969}    &            & 210                    & L10 (1)                       & 40                    & L10 (1)                       \\
                          &                          &          &                         &            & 221                    & (34)                          & 70                    & (25)                          \\ \hline
\multirow{2}{*}{21}       & 3\,876                     & M        & 969                     & M          & 210                    & M                             & \multirow{2}{*}{190}  &                               \\
                          & 5\,601                     & (428)    & 1\,290                    & (103)      & 308                    & (54)                          &                       &                               \\ \hline
\multirow{2}{*}{22}       & 5\,544                     & L10 (2)  & 1\,008                    & L10 (2)    & 210                    & M                             & 190                   & M                             \\
                          & 8\,844                     & (888)    & 1\,849                    & (208)      & 454                    & (87)                          & 227                   & (20)                          \\ \hline
\multirow{2}{*}{23}       & 5\,544                     & M        & 1\,008                    & M          & 210                    & M                             & 190                   & M                             \\
                          & 15\,355                    & (1\,469)   & 2\,808                    & (366)      & 751                    & (134)                         & 291                   & (36)                          \\ \hline
\multirow{2}{*}{24}       & \multirow{2}{*}{33\,649}   &          & 3\,366                    & L10 (2)    & \multirow{2}{*}{1\,771}  &                               & 190                   & M                             \\
                          &                          &          & 4\,734                    & (579)      &                        &                               & 384                   & (58)                          \\ \hline
\multirow{2}{*}{25}       & 33\,649                    & M        & \multirow{2}{*}{10\,626}  &            & 1\,771                   & M                             & 190                   & M                             \\
                          & 49\,605                    & (2\,971)   &                         &            & 2\,271                   & (144)                         & 525                   & (92)                          \\ \hline
\multirow{2}{*}{26}       & 40\,898                    & L10 (3)  & 10\,626                   & M          & 1\,771                   & M                             & 286                   & L11 (2)                       \\
                          & 78\,927                    & (6\,343)   & 14\,514                   & (834)      & 3\,071                   & (285)                         & 762                   & (144)                         \\ \hline
\multirow{2}{*}{27}       & 40\,898                    & M        & 10\,626                   & M          & 1\,771                   & M                             & 286                   & M                             \\
                          & 137\,410                   & (10\,595)  & 21\,020                   & (1\,750)     & 4\,311                   & (494)                         & 1\,242                  & (220)                         \\ \hline
\multirow{2}{*}{28}       & \multirow{2}{*}{296\,010}  &          & 16\,016                   & L11 (3)    & 4\,140                   & L10 (2)                       & \multirow{2}{*}{2\,925} &                               \\
                          &                          &          & 32\,169                   & (3\,150)     & 6\,408                   & (818)                         &                       &                               \\ \hline
\multirow{2}{*}{29}       & 296\,010                   & M        & 16\,830                   & L7 (5, 24) & 4\,140                   & M                             & 2\,925                  & M                             \\
                          & 442\,270                   & (21\,745)  & 54\,342                   & (5\,035)     & 10\,606                  & (1\,269)                        & 3\,643                  & (187)                         \\ \hline
\multirow{2}{*}{30}       & 621\,075                   & L10 (3)  & \multirow{2}{*}{118\,755} &            & \multirow{2}{*}{23\,751} &                               & 3\,003                  & L10 (1)                       \\
                          & 707\,796                   & (47\,420)  &                         &            &                        &                               & 4\,723                  & (366)                         \\ \hline
\multirow{2}{*}{31}       & 621\,075                   & M        & 118\,755                  & M          & 23\,751                  & M                             & 3\,003                  & M                             \\
                          & 1\,234\,969                  & (79\,818)  & 164\,701                  & (7\,327)     & 31\,093                  & (1\,389)                        & 6\,291                  & (615)                         \\ \hline
\multirow{2}{*}{32}       & \multirow{2}{*}{2\,629\,575} &          & 139\,568                  & L10 (2)    & 33\,600                  & L10 (2)                       & 4\,800                  & L10 (2)                       \\
                          &                          &          & 240\,248                  & (15\,849)    & 42\,433                  & (2\,876)                        & 8\,682                  & (999)                         \\ \hline
\multirow{2}{*}{33}       & 2\,629\,575                  & M        & 139\,568                  & M          & 33\,600                  & M                             & 4\,800                  & M                             \\
                          & 3\,966\,925                  & (165\,264) & 369\,680                  & (29\,044)    & 60\,038                  & (5\,113)                        & 12\,696                 & (1\,601)                        \\ \hline
\end{tabular}
\caption{Lower and upper bounds on $\SP(n,k)$}
\label{tab:smallBounds}
\end{table}

\setlength{\abovecaptionskip}{0pt plus 1pt minus 1pt}

\begin{figure}[H]
\centering
 \includegraphics[width = \textwidth]{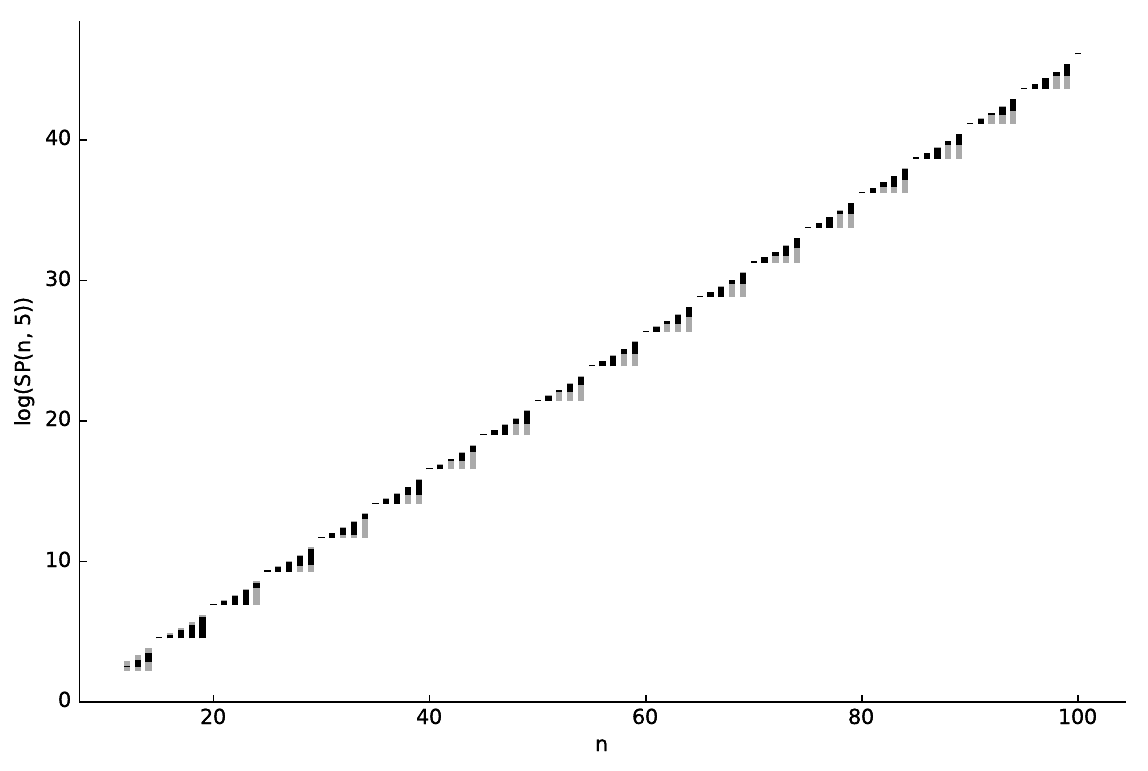}
\caption{Best known bounds on $\SP(n,5)$ compared to $\NLB(n,5)$ and $\MMS(n,5)$}
\label{fig:k5}
\end{figure}\bigskip

\begin{figure}[H]
\centering
 \includegraphics[width = \textwidth]{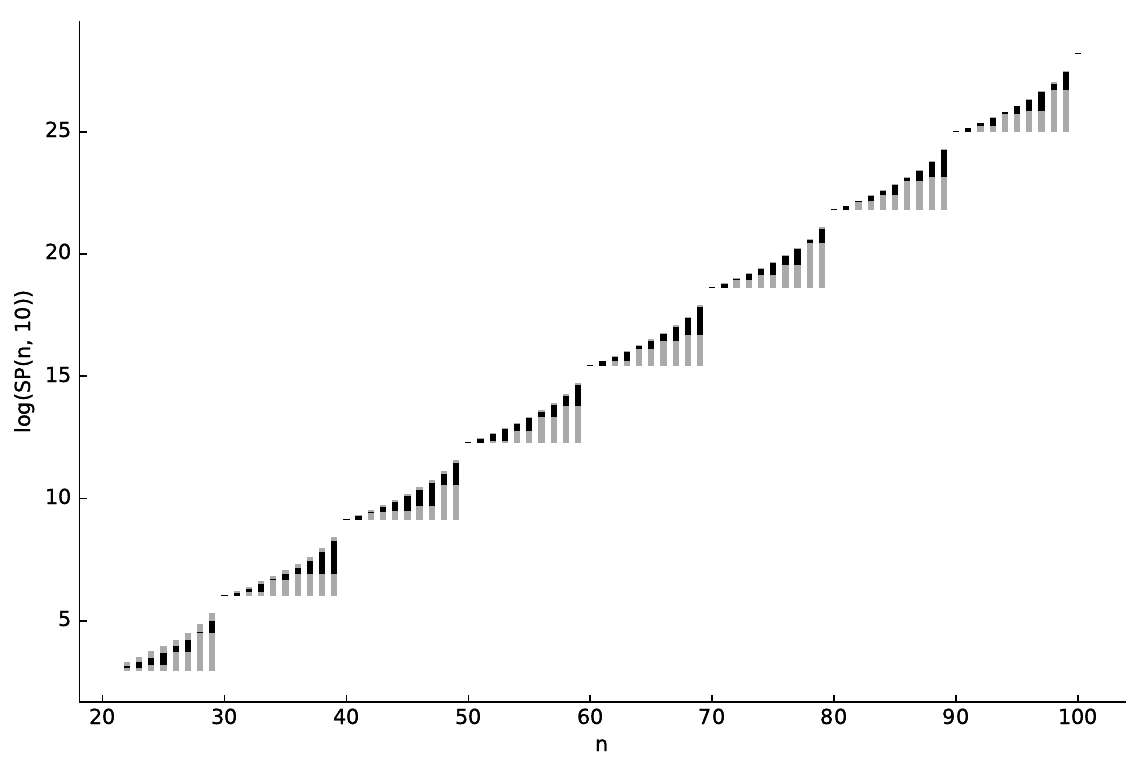}
\caption{Best known bounds on $\SP(n,10)$ compared to $\NLB(n,10)$ and $\MMS(n,10)$}
\label{fig:k10}
\end{figure}

\section*{Acknowledgments}\label{ackref}
Much of this research was undertaken during visits by C.J.~Colbourn and D.~Horsley to Beijing Jiaotong University.
They express their sincere thanks to the 111 Project of China (B16002) for financial support and to the Department of Mathematics at Beijing Jiaotong University for their kind hospitality. The authors' research received support from the following sources. Y.~ Chang: NSFC grant 11971053; C.J.~Colbourn: NSF grants 1421058 and 1813729; A.~Gowty: Australian Government Research Training Program Scholarship; D.~Horsley: ARC grants DP150100506 and FT160100048; J.~Zhou: NSFC grant 11571034.

\end{document}